\documentclass[12pt]{amsart}
\usepackage{amsmath,amsthm}
\usepackage{amssymb}
\usepackage{graphicx}
\usepackage[all]{xypic}
\SelectTips{cm}{12} \UseTips{}
\usepackage{euscript}
\usepackage{amsfonts}
\usepackage{latexsym}
\usepackage{euscript}
\usepackage{epsf}
\theoremstyle{plain} \swapnumbers
    \newtheorem{thm}{Theorem}[section]

    \newtheorem{prop}[thm]{Proposition}
    \newtheorem{lemma}[thm]{Lemma}
    
    \newtheorem{corollary}[thm]{Corollary}

    \newtheorem{subsec}[thm]{}
\theoremstyle{definition}
    \newtheorem{defn}[thm]{Definition}
    \newtheorem{example}[thm]{Example}
    
    \newtheorem{notation}[thm]{Notation}
    \newtheorem{summary}[thm]{Summary}
    \newtheorem{ack}[thm]{Acknowledgements}
\theoremstyle{remark}
        \newtheorem{remark}[thm]{Remark}

\newenvironment{myeq}[1][]
{\stepcounter{thm}\begin{equation}\tag{\thethm}{#1}}
{\end{equation}}

\newcommand{\mysdg}[2][]{\myeq[#1]\xymatrix@R=25pt@C=15pt{#2}}
\newcommand{\myadiagnum}[2][]
{\stepcounter{thm}\begin{equation}
     \tag{\thethm}{#1}\vcenter{\xymatrix@R=20pt@C=25pt{#2}}\end{equation}}
\newcommand{\myadiagnumm}[2][]
{\stepcounter{thm}\begin{equation}
     \tag{\thethm}{#1}\vcenter{\xymatrix@R=15pt@C=32pt{#2}}\end{equation}}
\newcommand{\myaaadiag}[2][]
{\stepcounter{thm}\begin{equation}
     \tag{\thethm}{#1}\vcenter{\xymatrix@R=10pt@C=15pt{#2}}\end{equation}}
\newcommand{\mybbbdiag}[2][]
{\stepcounter{thm}\begin{equation}
     \tag{\thethm}{#1}\vcenter{\xymatrix@R=13pt@C=17pt{#2}}\end{equation}}
\newcommand{\mycccdiag}[2][]
{\stepcounter{thm}\begin{equation}
     \tag{\thethm}{#1}\vcenter{\xymatrix@R=15pt@C=30pt{#2}}\end{equation}}
\newcommand{\mydiagram}[2][]
{\stepcounter{thm}\begin{equation}
     \tag{\thethm}{#1}\vcenter{\xymatrix{#2}}\end{equation}}
\newcommand{\myrdiag}[2][]
{\stepcounter{thm}\begin{equation}
     \tag{\thethm}{#1}\vcenter{\xymatrix@R=20pt@C=30pt{#2}}\end{equation}}
\newcommand{\myssdg}[2][]
{\stepcounter{thm}\begin{equation}
     \tag{\thethm}{#1}\vcenter{\xymatrix@R=15pt@C=15pt{#2}}\end{equation}}
\newcommand{\mykkdiag}[2][]
{\stepcounter{thm}\begin{equation}
     \tag{\thethm}{#1}\vcenter{\xymatrix@R=2pt@C=13pt{#2}}\end{equation}}
\newcommand{\mysdiag}[2][]
{\stepcounter{thm}\begin{equation}
     \tag{\thethm}{#1}\vcenter{\xymatrix@R=15pt@C=15pt{#2}}\end{equation}}
\newcommand{\mytdiag}[2][]
{\stepcounter{thm}\begin{equation}
     \tag{\thethm}{#1}\vcenter{\xymatrix@R=35pt@C=65pt{#2}}\end{equation}}
\newcommand{\myudiag}[2][]
{\stepcounter{thm}\begin{equation}
     \tag{\thethm}{#1}\vcenter{\xymatrix@R=19pt@C=19pt{#2}}\end{equation}}
\newcommand{\myvdiag}[2][]
{\stepcounter{thm}\begin{equation}
     \tag{\thethm}{#1}\vcenter{\xymatrix@R=35pt@C=50pt{#2}}\end{equation}}
\newcommand{\mydiagrm}[2][]
{\stepcounter{thm}\begin{equation}
    \tag{\thethm}{#1}\vcenter{\entrymodifiers={+++[o]}\xymatrix@R=25pt@C=25pt{#2}}\end{equation}}
\newcommand{\myrrdiag}[2][]
{\stepcounter{thm}\begin{equation}
     \tag{\thethm}{#1}\vcenter{\xymatrix@R=25pt@C=35pt{#2}}\end{equation}}
 \newcommand{\myrrrdiag}[2][]
{\stepcounter{thm}\begin{equation}
     \tag{\thethm}{#1}\vcenter{\xymatrix@R=15pt@C=55pt{#2}}\end{equation}}
\newcommand{\sect}{\setcounter{thm}{0}\section}
\newcommand{\supsect}[2]
{\vspace*{-5mm}\quad\\\begin{center}\textbf{{#1}}\vsm.~~~~\textbf{{#2}}\end{center}}
\newenvironment{mysubsection}[2][]
{\begin{subsec}\begin{upshape}\begin{bfseries}{#2.}
\end{bfseries}{#1}}
{\end{upshape}\end{subsec}}
\newcommand{\stk}[1]{\stackrel{#1} {\longrightarrow}}

\newcommand{\w}[2][ ]{\ \ensuremath{#2}{#1}\ }

\newcommand{\wwb}[1]{\ (\ensuremath{#1})-}
\newcommand{\ww}[2][ ]{\ \ensuremath{#2}{#1}}
\newcommand{\wh}{\ -- \ }

\newcommand{\wb}[2][ ]{\ (\ensuremath{#2}){#1}\ }
\newcommand{\wref}[2][ ]{\ \eqref{#2}{#1}\ }

\newcommand{\hsm}{\hspace{2 mm}}

\newcommand{\vsm}{\vspace{2 mm}}
\newcommand{\vsn}{\vspace{1 mm}}
\newcommand{\po}{\ar@{}[dr]|{\text{\pigpenfont R}}}
\newcommand{\pb}{\ar@{}[dr]|{\text{\pigpenfont J}}}

\newcommand{\hra}{\hookrightarrow}

\newcommand{\lra}[1]{\langle{#1}\rangle}
\newcommand{\llra}[1]{\langle\langle{#1}\rangle\rangle}

\newcommand{\nn}{\mathbb{N}}

\newcommand{\cof}{\operatorname{cof}}
\newcommand{\hcof}{\operatorname{hcof}}
\newcommand{\cf}[3][2]{{\cof}^{(#2)}{(#3)}}
\newcommand{\hcf}[3][2]{{\hcof}^{(#2)}{(#3)}}

\newcommand{\cub}{\operatorname{cub}}

\newcommand{\ho}{\operatorname{ho}}
\newcommand{\Hom}{\operatorname{Hom}}

\newcommand{\Id}{\operatorname{Id}}

\newcommand{\init}{\operatorname{init}}

\newcommand{\rec}{\operatorname{rec}}

\newcommand{\Rec}{\operatorname{Rec}}

\newcommand{\Final}{\operatorname{Fin}}
\newcommand{\fs}{f\sb{\ast}}
\newcommand{\gs}{g\sb{\ast}}
\newcommand{\ks}{k\sb{\ast}}

\newcommand{\FFs}{F\sb{\ast}}
\newcommand{\Xs}{X\sb{\ast}}
\newcommand{\Ys}{Y\sb{\ast}}
\newcommand{\Zs}{Z\sb{\ast}}

\newcommand{\Ws}{W\sb{\ast}}

\newcommand{\Dd}{\mathcal{D}}

\newcommand{\XF}{(\Xs,\FFs)}
\newcommand{\XFn}{(\Xs,\widetilde{F}_{\ast})}
\newcommand{\XGn}{(X'_{\ast},\widetilde{G}_{\ast})}
\newcommand{\XG}{(\Xs',G_{\ast})}
\newcommand{\rk}{\vec{r}_k}
\newcommand{\an}{\widetilde{\alpha}}
\newcommand{\bn}{\widetilde{\beta}}
\newcommand{\ann}[3][2]{{\widetilde{\alpha}}^{(#2)}_{#3}}
\newcommand{\bnn}[3][2]{{\widetilde{\beta}}^{(#2)}_{#3}}

\newcommand{\Fm}[3][2]{F^{(#2)}_{#3}}
\newcommand{\Fn}[3][2]{\widetilde{F}^{(#2)}_{#3}}

\newcommand{\Gm}[3][2]{G^{(#2)}_{#3}}
\newcommand{\Gn}[3][2]{\widetilde{G}^{(#2)}_{#3}}
\newcommand{\Tm}[3][2]{T^{(#2)}_{#3}}
\newcommand{\Tn}[3][2]{\widetilde{T}^{(#2)}_{#3}}

\newcommand{\VV}[1]{W(#1)}
\newcommand{\Ac}{\mathbb{A}}
\newcommand{\Bc}{\mathbb{B}}
\newcommand{\In}[3][2]{\mathcal{I}^{#2}_{#3}}
\newcommand{\Ln}[3][2]{\mathcal{L}^{#2}_{#3}}

\newcommand{\xm}[3][2]{\mathbb{X}^{(#2)}_{#3}}
\newcommand{\Mm}[3][2]{\mathbb{M}^{(#2)}_{#3}}
\newcommand{\CCm}[3][2]{{\mathbb{C}}^{(#2)}{#3}}
\newcommand{\CCn}[3][2]{\widetilde{\mathbb{C}}^{(#2)}{#3}}
\newcommand{\xn}[3][2]{\widetilde{\mathbb{X}}^{(#2)}_{#3}}
\newcommand{\Mn}[3][2]{\widetilde{\mathbb{M}}^{(#2)}_{#3}}
\newcommand{\AAA}[3][2]{\mathfrak{A}^{(#2)}_{#3}}
\newcommand{\BBB}[3][2]{\mathfrak{B}^{(#2)}_{#3}}
\newcommand{\AAAn}[3][2]{\widetilde{\mathfrak{A}}^{(#2)}_{#3}}
\newcommand{\BBBn}[3][2]{\widetilde{\mathfrak{B}}^{(#2)}_{#3}}
\newcommand{\FFF}[3][2]{{\mathfrak{F}}^{(#2)}_{#3}}
\newcommand{\GGG}[3][2]{{\mathfrak{G}}^{(#2)}_{#3}}
\newcommand{\pa}[3][2]{{\partial}^{#2}_{[#3]}}

\newcommand{\Fin}[3][2]{{\Final}^{(#2)}{(#3)}}

\newcommand{\vartha}[1]{\vartheta\sb{#1}}
\newcommand{\RR}[3][2]{{\mathfrak{R}}^{(#2)}{#3}}
\newcommand{\ex}{\Sigma X}

\newcommand{\ee}[3][2]{\widetilde{\Sigma}^{#2}{#3}}
\newcommand{\eex}{\widetilde{\Sigma} X}

\newcommand{\RRR}{\widetilde{R}}
\newcommand{\xfygz}{X \stk{f} Y \stk{g} Z}

\newcommand{\xfygzhw}{X \stk{f} Y \stk{g}  Z \stk{h} W}

\newcommand{\fhkg}{\raisebox{10pt}{\xymatrix@R=8pt@C=8pt{X \ar[r]^{h} \ar[d]_{f}  &  Z \ar[d]^{g}    \\
Y \ar[r]_{k} & W}}}
\newcommand{\fhkgs}{\xymatrix@R=8pt@C=8pt {  \Xs \ar[r]^{h_{\ast}} \ar[d]_{\fs}  &  \Zs \ar[d]^{\gs}    \\
\Ys \ar[r]_{\ks} & \Ws } }
\newcommand{\tos}{\approx}
\newcommand{\tosr}{\stackrel{\rec}{\approx}}
\newcommand{\toss}{\stackrel{\cub}{\approx}}
\newcommand{\vva}{\vec{a}}
\newcommand{\vvb}{\vec{b}}

\begin{document}

\title{Higher order Toda brackets}
\author{Azez Kharouf}

\maketitle
%

\begin{abstract}
We describe two ways to define higher order Toda brackets in a pointed simplicial model category $\Dd$:
one is a recursive definition using model categorical constructions, and the second uses the
associated simplicial enrichment. We show that these two definitions agree, by providing
a third, diagrammatic, description of the Toda bracket, and explain how it serves as the obstruction to
rectifying a certain homotopy-commutative diagram in $\Dd$.
\end{abstract}

\section*{Introduction} \label{cint}
The Toda bracket is an operation on homotopy classes of maps first defined by H.~Toda in \cite{TodG,TodC}
(see \S \ref{def 1 of 2}).
Toda used this construction to compute homotopy groups of spheres; Adams later showed (in \cite{AdHI}) that it
can also be used to calculate differentials in spectral sequences (see also \cite{HarpSC}).
The original definition was later extended to longer Toda brackets (see, e.g., \cite{GWalkL}).
In stable model categories, this can be done in several equivalent ways (see \cite{JCohDS,KocU,CFranH}) .

\begin{mysubsection}{The classical Toda bracket}\label{def 1 of 2}
Given maps \w{\xfygzhw} in a pointed model category $\Dd$, with nullhomotopies \w{F :CX \to Z } for \w{ g\circ f } and \w{G :CY \to W} for
\w[,] {h\circ g} from the pushout property in the diagram:
$$
\xymatrix@R=15pt {  X \ar@{^{(}->}[d]^{} \ar@{^{(}->}[r]^{}  & CX \ar@/^{1.5pc}/[ddr]^{h\circ F} \ar[d]   \\
CX \ar@/_{1.5pc}/[drr]_{G\circ Cf} \ar[r]  & \Sigma{X} \ar@{-->}[dr]^{T}  \\
& & W}
$$
The \emph{Toda bracket} is the induced map \w[.]{T:\Sigma X \to W}
\end{mysubsection}

\begin{mysubsection}{Alternative approaches}
More generally, higher homotopy operations have been defined abstractly (see \cite{SpanS,SpanH} and \cite{KlauT,MaunC}),
using two main approaches: the original definition for topological spaces generalizes to any pointed model
category (see \cite{BJTurnHA,BJTurnC,BBSenD}), while an alternative construction uses a (pointed) cubical enrichment
(see \cite{BJTurnHH,BBGondH}). Since cubically enriched categories are Quillen equivalent to simplicially enriched categories,
it should be possible to extend the second definition to any $(\infty,1)$-category (see \cite{BergnI}).
However, the connection between these two approaches had not been made clear so far.

In this article we study general Toda brackets in any suitable pointed model category $\Dd$. These can be defined in terms
of standard model-category constructions (stated here in terms of homotopy cofibers, although the translation to the
Eckmann-Hilton duals in terms of homotopy fibers is straightforward \wh see \cite{BJTurnC}). We call this the \emph{recursive} approach,
since the constructions used at the $n$-th stage depend on the previous stage (in fact, on the coherent
vanishing of all lower Toda brackets). We show that the vanishing of the Toda brackets so defined is the (last) obstruction
to lifting the corresponding "chain complex" from \w{\ho\Dd} to $\Dd$ (that is, making the successive composites strictly zero) \wh
see Theorem \ref{reduced rectifiyng} below and \cite{BBSenD}.

By \cite{DKanF}, every model category can be endowed with a simplicial enrichment having the same homotopy category;
in a simplicial model category, this can be done in a straightforward manner (see \cite[II, \S 2]{QuiH}).
Moreover, such a simplicial enrichment translates directly into a cubical enrichment, and conversely
(see, e.g., \cite{BJTurnHH}). In this paper, the \emph{cubical} approach to the general Toda brackets is not stated explicitly
in terms of a cubical enrichment, as in \cite{BBGondH}, but rather in terms of the simplicial
structure on $\Dd$ (in the sense of \cite[II, \S 1]{QuiH}).

Our main result is that these two approaches yield equivalent notions of higher Toda brackets
(see Theorems \ref{def1=def2} and \ref{Rev def1=def2}). To show this, we give a third \emph{diagrammatic} description of Toda brackets.
For example, the information needed for the construction in \S \ref{def 1 of 2}  can be encoded by the diagram:
\mydiagram[\label{eqfirstoda}]{ X  \ar@{^{(}->}[d] \ar[r]^{f}  & Y \ar[d]^{g} \ar@{^{(}->}[r]& CY \ar[d]^{G} \\
CX \ar[r]^{F}  & Z \ar[r]^{h}  & W. }
We think of this as a sequence of two horizontal maps of vertical $1$-cubes (in the general case, we will have two maps of $n$-cubes),
from which we obtain the associated Toda bracket by suitable homotopy colimits.
\end{mysubsection}

\begin{notation}\label{cone defn}
Throughout this paper $\Dd$ will denote a pointed proper simplicial model category (see \cite[II \S 2]{QuiH}).

The \emph{cone} of an object \w[,]{X\in\Dd} denoted by \w[,]{CX} is the (strict) cofiber
of \w[,]{i^X_{1}:X\to X \otimes I}
where  \w{X \otimes I} is the functorial cylinder object provided by the simplicial structure, we define
 \w{i^X:X \to CX} as in the following diagram:
$$
\xymatrix@R=15pt @C=30pt {
X \ar[d] \ar[r]^{i^X_1} & X \otimes I \ar[d]_{l^X} & X \ar[l]_{i_0^X} \ar[ld]^{i^X} \\
\ast \ar[r] & CX
}
$$
For all \w{m \in \nn} we set \w[,]{C^m X:=C(C^{m-1}X)} where \w[.]{C^0X:=X}

The \emph{cofiber} of any map \w{f:X \to Y} \w{\cof(f)} is the pushout of \w[,]{\ast\leftarrow X \xrightarrow{f} Y} with structure map \w[.]{r^f:Y\to\cof(f)}
Our standard model for the homotopy cofiber of a map \w[]{f:X \to Y} is the pushout of
\xymatrix@R=25pt { CX     & X \ar[r]^{f} \ar@{_{(}->}[l]^{}  & Y ,}
which we denote by \w[.]{\hcof(f)}
Functoriality of the pushout defines the natural map \w[,]{\zeta_f:\hcof(f)\to\cof(f)}
and if \w{f :X \to Y} is a cofibration, then right properness of the model category $\Dd$ implies that
\w{\zeta_f : \hcof(f) \to \cof(f)} is a weak equivalence.

The \emph{suspension} of \w[,]{X\in\Dd} denoted by \w[,]{\Sigma X} is defined to be the pushout of
\w{CX \xleftarrow{i^X}X\xrightarrow{i^X}CX} (that is, \w[).]{\hcof(i^X)}
Another version of the suspension of $X$ is the pushout of \w{\ast \leftarrow X\xrightarrow{i^X}CX}
(that is, \w[),]{\cof(i^X)} which we denote by \w[.]{\eex}

We have a natural weak equivalence \w[.]{\zeta_{X}:=\zeta_{i^X} : \ex \to \eex }
\end{notation}

\begin{mysubsection}{Organization}
In Section  \ref{OrdToda} we provide a recursive definition of the general Toda bracket of length $n$ in
a pointed model category $\Dd$, and of the data needed to define it (called a \emph{recursive Toda system}).
We then show that these Toda brackets serve as the last obstruction to rectifying certain diagrams.
In Section \ref{GTb} we give a new definition of the higher Toda bracket (and the data required, called a \emph{cubical Toda system})
in terms of the cubical (or simplicial) enrichment in $\Dd$ (as in \cite{BBGondH}). We then reinterpret this data in terms of certain
cubically-shaped diagrams, and show that the new Toda brackets can also be described in terms of certain homotopy colimits of the diagrams.
In Section \ref{DdTB} we  provide a diagrammatic description of the recursive construction, too, and use it to show how to pass from the
cubical to the recursive definitions of higher Toda brackets (see Theorem \ref{def1=def2}). In Section \ref{PRCD} we show how to pass from the recursive to the cubical constructions, yielding Theorem \ref{Rev def1=def2}.
\end{mysubsection}

\begin{ack}
This paper contains results from my Ph.D.\ thesis at the University of Haifa,  supported by Israel Science
Foundation grant 770/16 of my advisor, David Blanc.
\end{ack}

\sect{Model categorical approach to general Toda brackets} \label{OrdToda}
In this section we give the main (recursive) definition of the general Toda bracket. We then show how it
is related to the rectification of linear diagrams.

\supsect{\protect{\ref{OrdToda}}.A}{Ordinary Toda Bracket in terms of homotopy cofiber}

In order to understand better the definition of the Toda bracket, we require the following auxiliary constructions:

\begin{defn}\label{alphabeta}
Given two maps \w{\xfygz} in $\Dd$ with \w{g\circ f} nullhomotopic, for any
choice of nullhomotopy \w[,]{F : CX \to Z} from the commutative diagrams
\mysdiag[\label{eqnalpha}]{  X  \ar@{^{(}->}[d]^{} \ar@{^{(}->}[r]^{} \ar[rrd]^>>>>>>>{f} & CX \ar[d] \ar[rrd]^{F} & &  &&    X \ar@{^{(}->}[d]^{} \ar[r]^{f}  & Y \ar@/^{1.5pc}/[ddr]^{g} \ar[d]  \\
CX \ar[r] \ar[rrd]_>>>>>>>>>>>{Cf} & \ex \ar@{..>}[rrd]_<<<<<{\alpha}  & Y  \ar@{^{(}->}[d]^{} \ar[r]_{g}  &  Z \ar[d] &&  CX \ar@/_{1.5pc}/[drr]_{F} \ar[r]  & \hcof(f) \ar@{-->}[dr]^{\beta}  \\
& & CY \ar[r]  & \hcof(g) && & & Z }
we get the maps \w[]{\alpha = \alpha(f,g,F) : \Sigma X \to \hcof(g)}
and \w[.]{\beta = \beta(f,g,F) : \hcof(f) \to Z}
\end{defn}

\begin{defn}\label{perook 1}
Given \w{f,g,h,F,G} as in \S \ref{def 1 of 2}, note that the map $T$ defined there is just \w[.]{\beta(g,h,G) \circ \alpha(f,g,F)}
The collection of all homotopy classes of such maps $T$ (as $F$ and $G$ vary) is called the \emph{Toda bracket}, and traditionally
denoted by \w[.]{\lra{f,g,h}} The map $T$, denoted by \w[,]{\lra{f,g,h,(F,G)}} is called a \emph{value} of this Toda bracket.
\end{defn}

\begin{defn}\label{ealphabeta}
For \w{\xfygz} and \w{F : CX \to Z} as above, by functoriality of the strict cofiber we obtain maps \w{\widetilde{\alpha}(f,g,F)} and \w{\widetilde{\beta}(f,g,F)}
making the following diagram commute:
\mydiagram[\label{eqalphat}]{ X  \ar@{^{(}->}[d]_{i^X} \ar[r]^{f}  & Y \ar[d]^{g} \ar[r]& \cof(f) \ar[d]^>>>>{\widetilde{\beta}(f,g,F)} \\
CX \ar[r]^{F} \ar[d]  & Z \ar[r] \ar[d]  &  \cof(F)   \\
\eex \ar[r]^{\widetilde{\alpha}(f,g,F)}  &   \cof(g) }
\end{defn}

The \w{3\times3} Lemma then implies:

\begin{lemma}\label{cof(a)=cof(b)}
Given maps \w[,]{\xfygz}  with nullhomotopy \w[]{F} as in definition \ref{alphabeta},
$$
\cof(\widetilde{\alpha}(f,g,F))=\cof(\widetilde{\beta}(f,g,F))
$$
\end{lemma}
\begin{lemma}\label{equalalpha}
Given \w{\xfygz} and a nullhomotopy \w{F : CX \to Z} for \w{g\circ f} as above, we get the following commutative squares:
\mydiagram[\label{abconn}]{
\Sigma X \ar[rr]^{\alpha(f,g,F)} \ar[d]^{\zeta_X}_{\simeq} && \hcof(g) \ar[d]^{\zeta_g}  &  \hcof(f) \ar[d]^{\zeta_f} \ar[rr]^{\beta(f,g,F)}
&& Z \ar[d]^{r^{F}}\\
\eex \ar[rr]^{\widetilde{\alpha}(f,g,F)} && \cof(g) &  \cof(f) \ar[rr]^{\widetilde{\beta}(f,g,F)} && \cof(F)    }
If $f$, $g$, and $F$ are cofibrations, the vertical maps in \wref{abconn} are weak equivalences.
\end{lemma}

\begin{defn}\label{def 2 of 2}
Given \w{f,g,h,F,G} as in \S \ref{def 1 of 2}, where all maps and nullhomotopies are cofibrations, we
define \w[,]{\widetilde{T}=\llra{f,g,h,(F,G)}:=\bn(g,h,G) \circ \an(f,g,F)}
\end{defn}
\begin{prop}\label{def 1 = def 2 of 2}
Given \w{f,g,h,F,G} as in \S \ref{def 2 of 2}, we have a commutative diagram:
$$
\xymatrix@R=15pt @C=35pt{
\Sigma X \ar[rr]^{\lra{f,g,h,(F,G)}} \ar[d]^>>>{\simeq}_>>>{\zeta_X} && W \ar[d]_>>>{\simeq}^>>>>{r^{G}} \\
\widetilde{\Sigma} X \ar[rr]^{\lra{\lra{f,g,h,(F,G)}}} && \cof(G)    }
$$
In particular \w[.]{\lra{f,g,h,(F,G)} \tos\llra{f,g,h,(F,G)}}
\end{prop}

\begin{proof}
By Lemma \ref{equalalpha} and the composition of squares as in \wref[.]{abconn}
\end{proof}

\supsect{\protect{\ref{OrdToda}}.B}{Recursive definition of general Toda brackets}

We now define the general Toda bracket, in the spirit of Definition \ref{def 2 of 2}.
See also \cite[\S 5]{CFranH} and the sources cited there for a recursive definition in stable model categories.

\begin{defn}\label{rprojmodst}
If \w{\mathcal{I}} is a finite indexing category, \w{\Dd^{\mathcal{I}}} has two model category structures (injective and projective),
but both have the same weak equivalences, defined object-wise. Thus two diagrams $X$ and $Y$ are weakly equivalent (written \w[)]{X\tos Y}
if they are connected by a zigzag of weak equivalences in \w[.]{\Dd^{\mathcal{I}}}

We shall be interested in the \emph{projective model structure}, in which the fibrations are also
defined objectwise (see \cite[Theorem 5.1.3]{HovM}). The cofibrations are more complicated to describe, in general. However, when $\mathcal{I}$
is a partially ordered set filtered by subcategories \w{F\sb{0}\subset F\sb{1}\subset\dotsc F\sb{n}=\mathcal{I}} with all non-identity arrows strictly increasing
filtration, a cofibrant diagram $\Ac$ in \w{\Dd\sp{\In{}{}}} may be described recursively by requiring that:
\begin{enumerate}
    \item $\Ac(i)$ be cofibrant in $\Dd$ for each object \w[;]{i\in F\sb{0}}
    \item Assuming by induction that \w{\Ac|\sb{F\sb{k}}} is cofibrant, we require that the map from the colimit of \w{\Ac|\sb{F\sb{k}}}
    filtration $k$ to each object in \w{F\sb{k+1}\setminus F\sb{k}} be a cofibration.
\end{enumerate}
Such a diagram will be called \emph{strongly cofibrant}.
\end{defn}

\begin{example}\label{cofsqu}
The outer commutative square in

\mykkdiag[\label{eqstcofsq}]{
 X \ar@{^{(}->}[rrr]^{h} \ar@{^{(}->}[ddd]^{f} &&& Z \ar@{^{(}->}[ddd]^{g} \ar[ldd]_{s} \\
 && \\
&& P \ar@{^{(}->}[rd]^{a} & \\
Y \ar[urr]^{t}  \ar@{^{(}->}[rrr]^{k} &&& W
}
is strongly cofibrant if and only if $X$ is cofibrant, the maps $f$ and $h$ are cofibrations,
and the induced map $a$ from the pushout $P$ is a cofibration. Thus we see that in particular $g$ and $k$ are also cofibrations.
\end{example}
\begin{lemma}\label{cofibrant square induces cofibration}
If the outer commutative square in \wref{eqstcofsq} is strongly cofibrant,
the induced map from \w{\cof(f) \to \cof(g)} is a cofibration.
\end{lemma}

\begin{mysubsection}{Recursive definition of the Toda bracket}\label{Gen defn 2 of 2}
Given a linear diagram \w{\Xs=(\xymatrix{X_1 \ar@{^{(}->}[r]^{f_1}  & X_2 \ar@{^{(}->}[r]^{f_2} & \dots X_{n+2} \ar@{^{(}->}[r]^{f_{n+2}} & X_{n+3}})}  in \w{\mathcal{D}} in which the maps are cofibrations,
and let \w{\Fn{1}{j}} be a nullhomotopy for \w[,]{f_{j+1} \circ f_j}
such that the square:
$$
\xymatrix@R=20pt {  X_j  \ar@{^{(}->}[d]_{i^{X_j}} \ar@{^{(}->}[r]^{f_j}  & X_{j+1} \ar@{^{(}->}[d]^{f_{j+1}} \\
CX_j \ar@{^{(}->}[r]^{\Fn{1}{j}}  & X_{j+2} }
$$
is strongly cofibrant. Let \w{\ann{1}{j}:=\widetilde{\alpha}(f_j,f_{j+1},\Fn{1}{j}): \ee{}{X_j}\hra \cof{(f_{j+1})}}
and \w{\bnn{1}{j}:=\widetilde{\beta}(f_j,f_{j+1},\Fn{1}{j}): \cof{(f_j)} \hookrightarrow \cof(\Fn{1}{j})}
be as in Definition \ref{ealphabeta} \wh so both are cofibrations by Lemma \ref{cofibrant square induces cofibration}.
Finally, assume that for every \w{1 \leq k \leq n-2} and \w{1\leq j \leq n-k+1} we have a strongly cofibrant square
\mydiagram[\label{eqdefalpha}]{  \ee{k}{X_j}  \ar@{^{(}->}[d]_{i^{\ee{k}{X_j}}} \ar@{^{(}->}[r]^{\ann{k}{j}}  & \cof(\bnn{k-1}{j+1}) \ar@{^{(}->}[d]^{\bnn{k}{j+1}} \\
C\ee{k}{X_j} \ar@{^{(}->}[r]^{\Fn{k+1}{j}}  & \cof(\Fn{k}{j+1})  }
(where \w{\Fn{k+1}{j}} is a nullhomotopy for the composition \w[).]{\bnn{k}{j+1} \circ \ann{k}{j}}
We then define:
\begin{myeq}\label{eqalphbet}
\begin{cases}
\ann{k+1}{j}=\widetilde{\alpha}(\ann{k}{j},\bnn{k}{j+1},\Fn{k+1}{j}): \widetilde{\Sigma}^{k+1}{X_j} \hookrightarrow \cof{(\bnn{k}{j+1})}\\
\bnn{k+1}{j}=\widetilde{\beta}(\ann{k}{j},\bnn{k}{j+1},\Fn{k+1}{j}): \cof{(\ann{k}{j})} \hookrightarrow  \cof(\Fn{k+1}{j})
\end{cases}
\end{myeq}
(See Example \ref{examptoda2} below).

Since \w{\cof{(\bnn{k}{j+1})}=\cof{(\ann{k}{j+1})}} by Lemma \ref{cof(a)=cof(b)}, we can ask if the composites
\w{ \bnn{k+1}{j+1} \circ \ann{k+1}{j} } are nullhomotopic. We may therefore assume recursively that we have the strongly cofibrant squares:
$$
\xymatrix@R=25pt {  \ee{n-1}{X_1} \ar@{^{(}->}[d]_{i^{\ee{n-1}{X_1}}} \ar@{^{(}->}[r]^{\ann{n-1}{1}}
& \cof(\bnn{n-2}{2}) \ar@{^{(}->}[d]^{\bnn{n-1}{2}}    &&
\ee{n-1}{X_2} \ar@{^{(}->}[d]_{i^{\ee{n-1}{X_2}}} \ar@{^{(}->}[r]^{\ann{n-1}{2}}
& \cof(\bnn{n-2}{3}) \ar@{^{(}->}[d]^{\bnn{n-1}{3}} \\
C \ee{n-1}{X_1} \ar@{^{(}->}[r]^{\Fn{n}{1}}  & \cof(\Fn{n-1}{2})  &&
C \ee{n-1}{X_2} \ar@{^{(}->}[r]^{\Fn{n}{2}}  & \cof(\Fn{n-1}{3})
}
$$

This allows us to obtain \w{\ann{n}{1}} and \w{\bnn{n}{2}} by Definition \ref{ealphabeta}.
We define the \emph{value} \w{\lra{\lra{f_1,f_2,\dots,f_{n+2},\{\{\Fn{m}{k}\}_{k=1}^{n-m+2}\}_{m=1}^{n}}}}
of the corresponding \emph{recursive Toda bracket} to be the composite
\w[.]{\bnn{n}{2} \circ \ann{n}{1} : \ee{n}{X_1} \to \cof(\Fn{n}{2})}

Note that the maps \w{\ann{k}{j}} and \w{\bnn{k}{j}} depend not only on \w[,]{\widetilde{F}_j^{(k)}} but also on all previous
choices of lower nullhomotopies.
\end{mysubsection}

\begin{example}\label{examptoda2}
Given a linear diagram \w{\xymatrix@R=25pt { X_1 \ar@{^{(}->}[r]^{f_1}  & X_2 \ar@{^{(}->}[r]^{f_2} & X_{3}
\ar@{^{(}->}[r]^{f_{3}} & X_{4} \ar@{^{(}->}[r]^{f_{4}} & X_{5}
}}
of cofibrations, the first two maps \w{f_1} and \w{f_2} yield \w[,]{\ann{1}{1}:\widetilde{\Sigma}{X_1}\to\cof(f_2)}
the next two \w{f_2} and \w{f_3} yield \w{\bnn{1}{2}:\cof(f_2)\to\cof(\widetilde{F}_2^{(1)})} and
\w[,]{\ann{1}{2}:\widetilde{\Sigma}{X_2}\to\cof(f_3)} and so on. Thus the stages in our recursive process consist of:
\begin{enumerate}
\renewcommand{\labelenumi}{\quad~Step \arabic{enumi}.~}
\item \w{\widetilde{\Sigma}{X_1}\stk{\ann{1}{1}}\cof(f_2)\stk{\bnn{1}{2}} \cof(\widetilde{F}_2^{(1)}) \quad \text{and} \quad
\widetilde{\Sigma}{X_2}\stk{\ann{1}{2}}\cof(f_3)\stk{\bnn{1}{3}}   \cof(\widetilde{F}_3^{(1)})}
\item \w{\widetilde{\Sigma}^2{X_1}\stk{\ann{2}{1}}\cof(\bnn{1}{2})=\cof(\ann{1}{2})\stk{\bnn{2}{2}}  \cof(\widetilde{F}_2^{(2)})}
\end{enumerate}
\noindent So the value of the length $4$ recursive Toda bracket we obtain is the composite
$$
\lra{\lra{f_1,f_2,f_3,f_4,\Fn{1}{1},\Fn{1}{2},\Fn{1}{3},\Fn{2}{1},\Fn{2}{2}}} = \bnn{2}{2} \circ \ann{2}{1}
$$
\end{example}

\begin{defn}\label{dredchcx}
Given a linear diagram
\w{\Xs=(\xymatrix{X_1 \ar@{^{(}->}[r]^{f_1}  & X_2 \ar@{^{(}->}[r]^{f_2} & \dots X_{n+2} \ar@{^{(}->}[r]^{f_{n+2}} & X_{n+3}})
}
in $\Dd$ of length \w{n+2} and choices of nullhomotopies
\w{\widetilde{F}\sb{\ast}=\{\{\Fn{m}{k}\}_{k=1}^{n-m+2}\}_{m=1}^{n}}
as in Definition \ref{Gen defn 2 of 2}, we call the system \w{\XFn} an \emph{ $n$-th order recursive Toda system}
and denote \w{\lra{\lra{f_1,f_2,\dots,f_{n+2},\{\{\Fn{m}{k}\}_{k=1}^{n-m+2}\}_{m=1}^{n}}}} by
\w[.]{\Tn{n}{1}=\Tn{n}{1}\XFn}
\end{defn}

\begin{remark}\label{rredchcx}
Even if the maps \w{f_{j}} are not all cofibrations and the squares \wref{eqdefalpha} are not all strongly cofibrant, we can define maps \w{\ann{m}{k}=\ann{m}{k}\XFn} and \w{\bnn{m}{k}=\bnn{m}{k}\XFn} as in
\S \ref{Gen defn 2 of 2} and thus the corresponding recursive Toda brackets \w[]{\Tn{n}{1}=\Tn{n}{1}\XFn}
as above. However, we cannot expect these constructions to be homotopy meaningful in this case.
\end{remark}

%
\supsect{\protect{\ref{OrdToda}}.C}{Rectification of linear diagrams}
We show how the Toda bracket serves as the (last) obstruction
to rectifying certain diagrams in the model category $\Dd$.

\begin{defn}\label{rectify def}
A diagram \w{\Xs=(X_1\stk{f_1} X_2 \dots \dots X_{n-1} \stk{f_{n-1}} X_{n})} in a model category $\Dd$ with \w{f_{j+1}\circ f_{j}\sim\ast} for
each \w{1 \leq j \leq n-2} yields a chain complex in \w[.]{\ho\Dd} A \emph{rectification} of \w{\Xs} is a lifting of this chain complex to $\Dd$:
in other words, a diagram \w{\Xs'=(X'_1\stk{f'_1} X'_2\dots X'_{n-1} \stk{f'_{n-1}} X'_{n})} in $\Dd$, with \w[,]{\Xs'\tos\Xs} such that
\w{f'_{j+1}\circ f'_{j} = \ast} for every \w[.]{1 \leq j \leq n-2}
If such a rectification exists we say the original diagram is \emph{rectifiable}.
\end{defn}

\begin{thm}\label{reduced rectifiyng}
Given an $n$-th order recursive Toda system \w{\XFn} of length \w[,]{n+2} if \w{\Tn{n}{1}\XFn} is nullhomotopic, then
the underlying linear diagram \w{\Xs} is rectifiable.
\end{thm}
\begin{proof}
For each $k$ and $j$, \w{\XFn} provides a commuting diagram with horizontal cofibration sequences:
\mydiagram[\label{eqdfalpha}]{\ee{k}{X_j}  \ar@{^{(}->}[d]_{i^{\ee{k}{X_j}}} \ar@{^{(}->}[rr]^{\ann{k}{j}}  &&
\cof(\bnn{k-1}{j+1}) \ar@{^{(}->}[d]^{\bnn{k}{j+1}}\ar[rr]^{r^{\ann{k}{j}}}  && \cof(\ann{k}{j})\ar[d]^{\bnn{k+1}{j}}\\
C\ee{k}{X_j} \ar@{^{(}->}[rr]^{\Fn{k+1}{j}}  && \cof(\Fn{k}{j+1})  \ar[rr]^{r^{\Fn{k+1}{j}}}_{\simeq} && \cof(\Fn{k+1}{j})
}
Since \w{\cof{(\bnn{k-1}{j+1})}=\cof{(\ann{k-1}{j+1})}} by Lemma \ref{cof(a)=cof(b)}, flipping the second square along the diagonal yields
a commutative diagram:

$$
%
%
\xymatrix@R=20pt @C=12pt {  & & & & & X_{n+1} \ar@{^{(}->}[r]^{f_{n+1}} \ar[d]_{r^{f_n}}
& X_{n+2} \ar[d]^{\simeq}    \\
%
%
& & & & & \cof(f_n) \ar@{^{(}->}[r]^{\bnn{1}{n}} \ar@{.}[d]
& \cof(\Fn{1}{n}) \ar@{.}[d]    \\
%
%
%
& & & &    X_4 \ar[d] \ar@{.}[r] & \cof(\ann{n-4}{4}) \ar@{^{(}->}[r]^>>{\bnn{n-3}{4}} \ar[d]_{r^{\ann{n-3}{3}}} & \cof(\Fn{n-3}{4}) \ar[d]^{\simeq}   \\
%
%
& & X_3 \ar[d]  \ar@{^{(}->}[r]^{f_3}   & X_4 \ar@{=}[ru] \ar[d]^{\simeq}  \ar[r] & \cof(f_3) \ar[d] \ar@{.}[r]
&  \cof(\ann{n-3}{3}) \ar@{^{(}->}[r]^>>{\bnn{n-2}{3}} \ar[d]_{r^{\ann{n-2}{2}}} & \cof(\Fn{n-2}{3}) \ar[d]^{\simeq} \\
%
%
X_2 \ar[d]_{r^{f_1}} \ar@{^{(}->}[r]^{f_2}   & X_3 \ar@{=}[ru] \ar[d]^{\simeq} \ar[r] & \cof(f_2)
\ar[d]^{} \ar@{^{(}->}[r]^{\bnn{1}{2}}
& \cof(\Fn{1}{2}) \ar[r] \ar[d]^{\simeq}  & \cof(\ann{1}{2}) \ar[d]  \ar@{.}[r]
& \cof(\ann{n-2}{2}) \ar@{^{(}->}[r]^>>{\bnn{n-1}{2}} \ar[d]_{r^{\ann{n-1}{1}}} & \cof(\Fn{n-1}{2}) \ar[d]^{\simeq} \\
%
%
\cof(f_1)  \ar@{^{(}->}[r]^{\bnn{1}{1}} & \cof(\Fn{1}{1}) \ar[r]^{r^{\bnn{1}{1}}} & \cof(\ann{1}{1})\ar@{^{(}->}[r]^{\bnn{2}{1}}
& \cof(\Fn{2}{1}) \ar[r]^{r^{\bnn{2}{1}}} &  \cof(\ann{2}{1})  \ar@{..}[r]
& \cof(\ann{n-1}{1}) \ar[r]^>>>{\bnn{n}{1}}  & \cof(\Fn{n}{1})} \\
$$
Thus (up to the relation $\tos$) we may replace the initial segment of \w{\Xs} (without \w[)]{f_{n+2}} by:
\myrdiag[\label{rectify seq}]{  X_1 \ar[r]^{f_1} & X_2 \ar[r]^{ \bnn{1}{1} \circ r^{f_1}} &
 \cof(\Fn{1}{1}) \ar[r]^>>>>{\bnn{2}{1} \circ r^{\bnn{1}{1}}}
&\dots \ar[r] & \cof(\Fn{n-1}{1}) \ar[r]^>>>>{\bnn{n}{1} \circ r^{\bnn{n-1}{1}} }  &
\cof(\Fn{n}{1})
}
Note that we have the following commutative diagram:
\mydiagram[\label{diag of rectify}]{
\ee{n-1}{X_1} \ar@{^{(}->}[d]_{\ann{n-1}{1}} \ar@{^{(}->}[r]^{}  & C\ee{n-1}{X_1}
\ar@{^{(}->}[d]^{\Fn{n}{1}} \ar[r] &
\ee{n}{X_1} \ar@{^{(}->}[d]^{\ann{n}{1}} \ar[rd]^{\Tn{n}{1} } \\
\cof(\ann{n-2}{2}) \ar@{^{(}->}[r]^{\bnn{n-1}{2}}  \ar[d]  &
\cof(\Fn{n-1}{2}) \ar[r] \ar[d]^{\simeq} & \cof(\bnn{n-1}{2}) \ar[d] \ar[r]^{\bnn{n}{2}}   & \cof(\Fn{n}{2}) \\
\cof(\ann{n-1}{1})  \ar[r]^{\bnn{n}{1}} &  \cof(\Fn{n}{1})
\ar[r]^>>>>{r^{\bnn{n}{1}}} & \cof(\bnn{n}{1})=\cof(\ann{n}{1}) \ar@{-->}[ru] }

Because \w[,]{\bnn{n}{2} \circ \ann{n}{1} = \Tn{n}{1}\sim \ast } we obtain a dashed
map \w{s : \cof(\ann{n}{1}) \to \cof(\Fn{n}{2})} in
\wref[,]{diag of rectify}) so we can continue (the right end of) the previous diagram as:
$$
%
%
\xymatrix@R=20pt @C=12pt {  &   &  X_{n+2} \ar@{^{(}->}[r]^{f_{n+2}} \ar[d]_{r^{f_{n+1}}}
& X_{n+3} \ar[d]^{\simeq}   \\
%
%
 X_{n+1} \ar@{^{(}->}[r]^{f_{n+1}} \ar[d]_{r^{f_n}}
& X_{n+2} \ar[r]^{r^{f_{n+1}}} \ar@{=}[ru] \ar[d]^{\simeq} &
\cof(f_{n+1}) \ar@{^{(}->}[r]^{\bnn{1}{n+1}} \ar[d]_{r^{\ann{1}{n}}}
& \cof(\Fn{1}{n+1}) \ar[d]^{\simeq} \\
%
%
%
%
 \cof(f_n) \ar@{^{(}->}[r]^{\bnn{1}{n}} \ar@{.}[d]
& \cof(\Fn{1}{n}) \ar[r] \ar@{.}[d]   &
 \cof(\ann{1}{n}) \ar[r]^>>{\bnn{2}{n}} \ar@{.}[d] & \cof(\Fn{2}{n}) \ar@{.}[d]   \\
%
%
 \cof(\ann{n-2}{2}) \ar[r]^>>{\bnn{n-1}{2}} \ar[d]_{r^{\ann{n-1}{1}}} & \cof(\Fn{n-1}{2}) \ar[r] \ar[d]^{\simeq} & \cof(\ann{n-1}{2}) \ar[r]^>>{\bnn{n}{2}}  \ar[d]_<<<{r^{\ann{n}{1}}}  & \cof(\Fn{n}{2})  \\
%
%
 \cof(\ann{n-1}{1}) \ar[r]^>>{\bnn{n}{1}}  & \cof(\Fn{n}{1})
 \ar[r]^{r^{\bnn{n}{1}}} & \cof(\ann{n}{1}) \ar[ru]_{s} \\
 }
$$
The rectification is completed by replacing \w{f_{n+2}} with \w[.]{s \circ r^{\bnn{n}{1}} }
\end{proof}

%
%
%
\sect{A cubical version of general Toda brackets} \label{GTb}
In this section we provide an alternative ``cubical'' definition of general Toda brackets, essentially formulated in terms of a topological
(or simplicial) enrichment of $\Dd$. We then describe a diagrammatic approach to this definition.

\supsect{\protect{\ref{GTb}}.A}{Cubical Definition of the Toda Bracket}
The new definition of the general Toda bracket corresponds to the version of \S \ref{def 1 of 2}.
Implicitly we are using an enrichment for $\Dd$ in cubical sets, but since we are only concerned with (higher) nullhomotopies,
we can use the simplified approach of \cite{BBGondH}.

We first fix some notation for cubical diagrams:

\begin{defn}\label{dcubicals}
The \emph{n-cube category} \w{\In{n}{}}  is the category of binary sequences \w{\vva=\vva^{(n)}}
of length $n$ with morphisms given by the obvious partial order induced by \w{0\leq 1} in each coordinate.
Thus there is a (unique) map \w{\vva \to \vvb} if and only if \w{\vvb} is obtained from \w{\vva} by
replacing some of the $0$'s by $1$'s.

We denote by \w{\deg(\vva)} the number of $1$'s in \w[,]{\vva} and by \w{\init(\vva)\geq 0} the length of the initial segment of $1$'s in \w[,]{\vva}
for example \w{\deg(1,1,0,1,0,0)=3} and \w[.]{\init(1,1,0,1,0,0)=2}
We write \w{\rk} for the sequence (of degree \w[)]{n-1} with a single $0$ in place $k$.

Thus the category \w{\In{3}{}} is the partially ordered set:
\myaaadiag[\label{I^3}]{    (0,0,0)  \ar[d]^{} \ar[r] \ar[rrd] & (0,1,0) \ar[d] \ar[rrd] & & \\
(1,0,0) \ar[r] \ar[rrd] & (1,1,0)  \ar[rrd] & (0,0,1) \ar[d] \ar[r]  &  (0,1,1) \ar[d] \\
& & (1,0,1) \ar[r] & (1,1,1)  }
\end{defn}
\begin{defn}\label{delln}
We denote by \w{\Ln{n}{}} the full subcategory of \w{\In{n+1}{}} obtained by omitting \w[.]{(1,\dots,1)}

Thus \w{\Ln{1}{}} is the corner of the \w{2\times 2} square \w{\In{2}{}} opposite \w[:]{(1,1)}
$$
\xymatrix@R=10pt @C=10pt { (0,0) \ar[r] \ar[d] & (0,1)  \\
(1,0)
}
$$
Similarly, \w[]{\Ln{2}{}} is obtained from the \w{2\times 2\times 2} cube of  \wref{I^3} omitting \w[.]{(1,1,1)}
\end{defn}
\begin{defn}\label{n-cube}
An \emph{$n$-cube in $\Dd$} is a functor \w[,]{\Ac  :\In{n}{} \to \Dd}
and an \emph{$n$-cube map} \w{\mathfrak{F}} is a natural transformation between $n$-cubes.
We denote by \w{\Dd\sp{\In{n}{}}} the category of all $n$-cubes in $\Dd$.
By convention \w[.]{\Dd\sp{\In{0}{}}:=\Dd}
\end{defn}

\begin{remark}
Observe that a $1$-cube in $\Dd$ is a map between objects in $\Dd$, and more generally
we may think of an $n$-cube map in $\Dd$ as an \wwb{n+1}cube in $\Dd$.
\end{remark}
\begin{defn}\label{CC^mX}
For any \w[,]{X\in\Dd} the $m$-cube \w{\mathbb{C}^{(m)}X } in $\Dd$ is defined by
\w{ \mathbb{C}^{(m)}X(\vva)=C^{\deg(\vva)}X} (see \S \ref{cone defn}), where if $\vec{b}$ is obtained from $\vva$ by replacing the
$0$ in the $j$-th slot by $1$ and $k$ is the number of $1$'s in $\vva$ (and in $\vec{b}$) before the $j$-th slot, then
\w{\mathbb{C}^{(m)}X(\vva)\to\mathbb{C}^{(m)}X(\vec{b})} is \w[.]{C\sp{\deg(\vva)-k}(i^{C^k X})}

Thus
\w[]{\mathbb{C}^{(2)}X} is
\mybbbdiag[\label{eqc2x}]{X \ar@{^{(}->}[d]_{i^X} \ar@{^{(}->}[r]^{i^X}  & CX \ar@{^{(}->}[d]^{C(i^X)}  \\
CX \ar@{^{(}->}[r]_{i^{CX}} & C^2 X  \\
}
\end{defn}
\begin{defn}\label{LX}
We define \w{\mathbb{L}^{(n)}X} to be the restriction of \w{\mathbb{C}^{(n+1)}X } to \w[,]{\Ln{n}{}}
and denote its colimit by \w[.]{L^n X}

Thus \w[,]{L^1 X=\Sigma{X}} and \w{L^2 X} is the colimit of \w[:]{\mathbb{L}^{(2)}X}
$$
\xymatrix@R=15pt  {  X  \ar@{^{(}->}[d]^{} \ar@{^{(}->}[r] \ar@{^{(}->}[rrd]
& CX \ar@{^{(}->}[d] \ar@{^{(}->}[rrd]^{i^{CX}} & & \\
CX \ar@{^{(}->}[r]^{i^{CX}} \ar@{^{(}->}[rrd]_{i^{CX}} & C^2 X   &
CX \ar@{^{(}->}[d]_>>>{C(i^X)} \ar@{^{(}->}[r]_{C(i^X)}  &  C^2 X  \\
& & C^2 X  &   }
$$
which is weakly equivalent to the double suspension \w[.]{\Sigma\sp{2}X}
\end{defn}

\begin{defn}\label{d_k}
Given \w[,]{X,Y \in \Dd} and\w[,]{1\leq k\leq m} let
$$
\delta\sb{k}:=C^{m-k}(i^{C^{k-1}X}):C^{m-1}X\to C^m X
$$
and write
\w{d_{k}=d^{X,Y}_{k}(m):\Hom_\mathcal{D}(C^m X,Y) \to \Hom_\mathcal{D}(C^{m-1} X,Y)} for
the induced map \w[.]{\delta\sb{k}\sp{\ast}}
Thus for \w{m=2} we have:
$$
\xymatrix@R=15pt @C=30pt {  &   CX  \ar@{^{(}->}[d]_{C(i^X)} \ar@/^{1.5pc}/[rdd]^{d_1(f)} \\
CX \ar@{^{(}->}[r]_{i^{CX}} \ar@/_{1.5pc}/[rrd]_{d_2(f)} & C^2X \ar[rd]^{f}  \\
&& Y
}
$$
\end{defn}
\begin{lemma}\label{bouhocom}
For any \w[]{f:C^n X\to Y} and \w{g:C^m Y\to Z} we have
$$
d_{k}(g \circ C^m(f)) = \begin{cases}
g   \circ C^m(d_{k}(f) )   &  1 \leq   k   \leq n      \\
d_{k-n}(g )  \circ C^{m-1}(f)  &  n+1 \leq k \leq m+n
\end{cases}
$$
\end{lemma}
\begin{defn}\label{order cahin}
Given a linear diagram \w{\Xs:=(X_1\stk{f_1} X_2 \dotsc\stk{f_{n-1}} X_{n} \stk{f_{n}} X_{n+1})} of
length $n$ in $\Dd$, an \emph{$m$-th order cubical Toda system} over \w{\Xs}(\w{m \leq n-2}) is a system \w{\XF=\XF^{(m,n)}} of
(higher order) nullhomotopies
$$
\FFs=\{\Fm{k}{j} \in \Hom(C^k X_j ,X_{j+k+1})~|\ 1\leq k \leq m,\
1 \leq j \leq n-k\}
$$
satisfying:
\begin{myeq}\label{eq2}
\begin{cases}
d_1(\Fm{k}{j})= \Fm{k-1}{j+1} \circ C^{k-1}(\Fm{0}{j})       \\
d_{2}(\Fm{k}{j})=\Fm{k-2}{j+2} \circ C^{k-2}( \Fm{1}{j})           \\
 \qquad   \vdots  \qquad  \qquad   \vdots          \\
d_{r}(\Fm{k}{j})=\Fm{k-r}{j+r} \circ C^{k-r}( \Fm{r-1}{j})          \\
 \qquad   \vdots  \qquad  \qquad   \vdots          \\
d_{k}(\Fm{k}{j})=\Fm{0}{j+k} \circ \Fm{k-1}{j}   \\
\end{cases}
\end{myeq}
where \w{\Fm{0}{j}=f_j} (see \cite{BBGondH}).
\end{defn}
\begin{defn}\label{Gen defn 1 of 2}
Let \w[]{\XF^{(n,n+2)}} be an $n$-th order cubical Toda system of length \w[.]{n+2}
By Lemma \ref{bouhocom} and \wref[,]{eq2} the maps
$$
\Fm{n-k+1}{k+1}\circ C^{n-k+1}(\Fm{k-1}{1}) :  \mathbb{L}^{(n)}X_1(\rk) = C^{n}X_1 \to X_{n+3}
 \quad  \quad (1 \leq k \leq n+1)
$$
satisfy
$$
d_{j-1}(\Fm{m-k-2}{k+1} \circ C^{m-k-2}(\Fm{k-1}{1}))
=d_{k}(\Fm{m-j-2}{j+1} \circ C^{m-j-2}(\Fm{j-1}{1}))
$$
so they induce a map which is called \emph{value of the Toda bracket}:
$$
\Tm{n}{1}\XF=\lra{f_1,f_2,\dots,f_{n+2},\{\{\Fm{m}{k}\}_{k=1}^{n-m+2}\}_{m=1}^{n}}: L^n X_1 \to X_{n+3}~.
$$
One can similarly define \w{\Tm{m}{j}=\Tm{m}{j}\XF} for suitable $m$ and $j$.
\end{defn}
\begin{example}
Given a linear diagram \w{\Xs=(X_1 \stk{f_1} X_2 \stk{f_2} X_3 \stk{f_3} X_4 \stk{f_4} X_5)} of length $4$ in $\Dd$,
a second order cubical Toda system over \w{\Xs} consists of three (first order) nullhomotopies
\w{\Fm{1}{1} :C X_1  \to X_{3}} for \w[,]{f_2 \circ f_1} \w{\Fm{1}{2} :C X_2  \to X_{4}} for \w[,]{f_3 \circ f_2}
and \w{\Fm{1}{3} :C X_3  \to X_{5} } for \w[,]{f_4 \circ f_3 } respectively, and two second-order nullhomotopies
\w{\Fm{2}{1} \in \Hom(C^2X_1 ,X_{4})} and \w{\Fm{2}{2} \in \Hom(C^2X_2 ,X_{4}) }  satisfying:
\w[,]{d_1(\Fm{2}{1})=\Fm{1}{2} \circ C f_1} \w[,]{d_1(\Fm{2}{2})=\Fm{1}{3} \circ C f_2}
\w[,]{d_2(\Fm{2}{1})=f_{3} \circ \Fm{1}{1}}  and \w[.]{d_2(\Fm{2}{2})=f_{4} \circ \Fm{1}{2}}

The value \w{\lra{f_1,f_2,f_3,f_4,\Fm{1}{1},\Fm{1}{2},\Fm{1}{3},\Fm{2}{1},\Fm{2}{2}}:L^2 X_1\to X_5}
is depicted by:

%
%
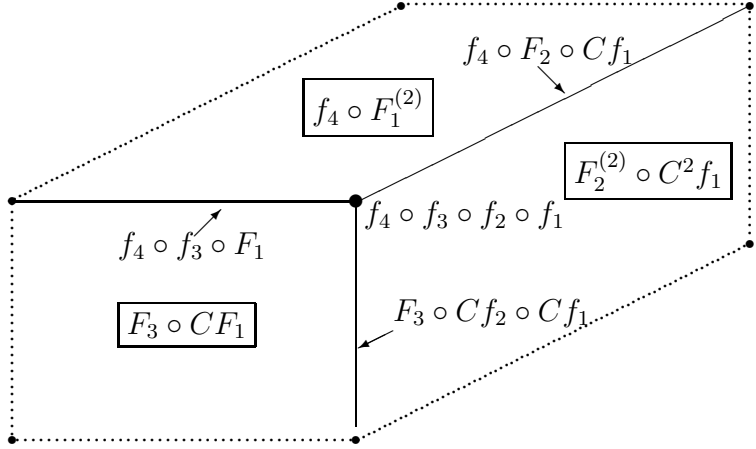
\begin{center}
\begin{picture}(350,180)(0,0)
%
%
\put(10,95){\circle*{3}}
\multiput(10,94)(0,-3){30}{\circle*{.5}}
{\put(50,75){$f_4\circ f_3\circ F_1$}}
{\put(79,82){\vector(1,1){10}}}
\put(140,95){\circle*{5}}
\put(140,95){\line(-1,0){130}}
\put(143,87){$f_4\circ f_3\circ f_2\circ f_1$}
\put(140,95){\line(0,-1){85}}
{\put(154,50){$F_3\circ Cf_2 \circ Cf_1$}}
{\put(154,46){\vector(-2,-1){13}}}
\put(140,5){\circle*{3}}
\multiput(12,5)(3,0){43}{\circle*{.5}}
\put(10,5){\circle*{3}}
{\put(50,45){\framebox{$F_3\circ CF_1$}}}
%
%
\put(157,169){\circle*{3}}
\multiput(157,169)(-2,-1){75}{\circle*{.5}}
\put(289,169){\circle*{3}}
\multiput(289,170)(-3,0){44}{\circle*{.5}}
\put(289,169){\line(-2,-1){150}}
{\put(180,149){$f_4 \circ F_2\circ Cf_1$}}
{\put(209,146){\vector(1,-1){10}}}
{\put(120,126){\framebox{$f_4 \circ \Fm{2}{1}$}}}
%
%
\multiput(289,170)(0,-3){30}{\circle*{.5}}
\put(289,79){\circle*{3}}
\multiput(289,79)(-2,-1){75}{\circle*{.5}}
{\put(219,102){\framebox{$\Fm{2}{2}\circ C^{2}f_1$}}}
\end{picture}
\end{center}
\end{example}
%
%
%

\supsect{\protect{\ref{GTb}}.B}{Diagrammatic interpretation for the cubical system}
We now describe cubical Toda systems in terms of diagrams in the
model category $\Dd$: more precisely, we will encode the Toda system by two maps of cubes, compassing the various nullhomotopies.
This will allow us to define a suitable equivalence relation $\tos$ on these systems (for which the Toda bracket is an invariant).

\begin{defn}\label{Mm}
Assume given an $n$-th order cubical Toda system \w[]{\XF} of length \w[.]{n+2}
For each \w{1 \leq m \leq n} and \w[,]{1 \leq j \leq n-m+3 }
we define an $m$-cube  \w{\Mm{m}{j}=\Mm{m}{j}\XF} in $\Dd$ by setting \w[,]{\Mm{m}{j} (\vva):= C^{\deg(\vva)-\init(\vva)} X_{j+\init(\vva)}}
(see Definition \ref{dcubicals}).

If \w{\vvb\in\In{n}{}} is obtained from $\vva$ by replacing $0$ in the $s$-th coordinate with $1$, the corresponding morphism
\w{\Mm{m}{j}(\vva)\to\Mm{m}{j}(\vvb)} is
$$
\begin{cases}
C^{\deg(\vva)-\init(\vva)-\ell} (F^{(\ell)}_{j+\init(\vva)}) & \text{if}\ s=\init(\vva)+1\\
C^{\deg(\vva)-r}(i^{C^{r-\init{\vva}} X_{j+\init(\vva)}}) & \text{otherwise}
\end{cases}
$$
where \w{\ell\geq 0} is the number of consecutive $1$'s starting from the \w{s+1} coordinate,
and \w[.]{r=a_1+\dots+a_{s-1}}

Thus \w{\Mm{1}{1}} is the $1$-cube \w[,]{f_1 :X_1 \to X_2} and \w{\Mm{3}{2}} is the $3$-cube:
$$
\xymatrix@R=25pt@C=40pt {  X_2  \ar[d]_{f_2} \ar@{^{(}->}[r]^{i^{X_2}} \ar@{^{(}->}[rrd]^>>>>>>>>>>>>{i^{X_2}} & CX_2 \ar[d]_>>>{F_2}
\ar@{^{(}->}[rrd]^{i^{CX_2}} & & \\
X_3 \ar[r]^{f_3} \ar@{^{(}->}[rrd]_{i^{X_3}} & X_4  \ar[rrd]_<<<<<<<<<<<<<<{f_4} & CX_2 \ar[d]^<<<{Cf_2} \ar@{^{(}->}[r]_{C(i^{X_2})}
&  C^2 X_2 \ar[d]^{\Fm{2}{2}} \\
& & CX_3 \ar[r]_{F_3} & X_5  }
$$
\end{defn}

\begin{defn}\label{AAA}
The $m$-cube map \w{\AAA{m}{j}=\AAA{m}{j}\XF : \mathbb{C}^{(m)}X_j \to \Mm{m}{j+1} } is given by
\w[.]{\AAA{m}{j} (\vva) = C^{\deg(\vva)-\init(\vva)} \Fm{\init(\vva)}{j}}

Thus \w{\AAA{2}{1} : \mathbb{C}^{(2)} X_1 \to \Mm{2}{2} } is the following back-to-front map of squares:
$$
\xymatrix@R =25pt@C=40pt {  X_1  \ar@{^{(}->}[d]_<<<{i^{X_1}} \ar@{^{(}->}[r]^{i^{X_1}} \ar[rrd]^>>>>>>>>>>>>>{f_1} & CX_1 \ar@{^{(}->}[d]^<<<{}
\ar[rrd]^{Cf_1} & & \\
CX_1 \ar@{^{(}->}[r]^{i^{CX_1}} \ar[rrd]_{F_1} & C^2 X_1  \ar[rrd]_<<<<<<<<{\Fm{2}{1}} & X_2 \ar[d]^<<<{f_2} \ar@{^{(}->}[r]_{i^{X_2}}
&  C X_2 \ar[d]^{F_2} \\
& & X_3 \ar[r]_{f_3} & X_4  }
$$
\end{defn}
\begin{defn}\label{xm}
Given an $n$-th order cubical Toda system \w[]{\XF} of length $n+2$ in $\Dd$, for each \w{1 \leq m \leq n} and \w{1 \leq j \leq n-m+1 }  we define the $m$-cube
\w{\xm{m}{j+m+2}=\xm{m}{j+m+2}\XF} by
$$
\xm{m}{j+m+2}(\vva) =
\begin{cases}
 X_{j+m+2}& \text{if}\ \vva=(1,\dots,1)  \\
C \Mm{m}{j+1}(\vva)& \text{otherwise}
\end{cases}
$$
with edge \w{\xm{m}{j+m+2}(\vva) \to \xm{m}{j+m+2}(\vvb)} equal to
$$
\begin{cases}
 \Fm{m-k+1}{j+k} & \text{if}\ \vva=\rk\ \text{and}\  \vvb=(1,\dots,1)\\
C (\Mm{m}{j+1}(\vva)\to  \Mm{m}{j+1}(\vvb))  & \text{otherwise}
\end{cases}
$$

Thus \w{\xm{1}{4}} is the $1$-cube \w[,]{F_2: CX_2 \to X_4}  and \w{\xm{3}{7}} is the $3$-cube:
$$
\xymatrix@R=25pt@C=45pt { C X_3  \ar[d]_{C f_3} \ar@{^{(}->}[r]^{C(i^{X_3})} \ar@{^{(}->}[rrd]^>>>>>>>>>>{C(i^{X_3})}
& C^2 X_3
 \ar[d]_{CF_3} \ar[rrd]^{C(i^{CX_3})} & & \\
CX_4 \ar[r]^{Cf_4} \ar[rrd]_{C(i^{X_4})} & C X_5  \ar[rrd]_<<<<<<<<<<<<<{F_5} & C^2 X_3 \ar[d]^<<<<<{C^2 f_3}
 \ar@{^{(}->}[r]_{C^2(i^{X_3})}
&  C^3 X_3 \ar[d]^{\Fm{3}{3}} \\
& & C^2X_4 \ar[r]_{\Fm{2}{4}} & X_7  }
$$
\end{defn}
\begin{defn}\label{BBB}
Given a higher order cubical Toda system \w[]{\XF} in $\Dd$, the $m$-cube map \w{\BBB{m}{j+1}=\BBB{m}{j+1}\XF : \Mm{m}{j+1} \to  \xm{m}{j+m+2}}  is
defined at the $\vva$-vertex by
$$
\BBB{m}{j+1}(\vva) =
\begin{cases}
 f_{j+m+1}& \text{if} ~  \vva=(1,1,\dots,1)  \\
i^{\Mm{m}{j+1}(\vva)}& \text{otherwise}
\end{cases}
$$

Thus \w{\BBB{2}{2}: \Mm{2}{2} \to \xm{2}{5}} is the back-to-front map of squares:
$$
\xymatrix@R=25pt@C=40pt {  X_2  \ar[d]_{f_2} \ar@{^{(}->}[r]^{i^{X_2}} \ar@{_{(}->}[rrd]^>>>>>>>>>>>>{i^{X_2}} & CX_2 \ar[d]_{F_2} \ar@{^{(}->}[rrd]^{i^{CX_2}} & & \\
X_3 \ar[r]^{f_3} \ar@{_{(}->}[rrd]_{i^{X_3}} & X_4  \ar[rrd]_<<<<<<<<<<<<<<{f_4} & CX_2 \ar[d],^<<<{Cf_2}
\ar@{^{(}->}[r]_{C(i^{X_2})}
&  C^2 X_2 \ar[d]^{\Fm{2}{2}} \\
& & CX_3 \ar[r]_{F_3} & X_5  }
$$
\end{defn}

\begin{summary}\label{stwocubes}
The data of an $n$-th cubical order Toda system \w[]{\XF} of length $n+2$ is encoded by the two of $n$-cube maps
\mydiagram[\label{eqtwocubes}]{
\mathbb{C}^{(n)}X_1 \ar[rrr]^{\AAA{n}{1} \XF} &&&  \Mm{n}{2}\XF
 \ar[rrr]^{\BBB{n}{2}\XF} &&& \xm{n}{n+3}\XF
}

Thus a second-order cubical Toda system\w{\XF^{(2)}} of length $4$ is described by:
$$
\xymatrix@R =25pt@C=30pt {  X_1  \ar@{^{(}->}[d] \ar@{^{(}->}[r] \ar[rrd]^>>>>>>>>>>{f_1} & CX_1 \ar@{^{(}->}[d]
\ar[rrd]^{Cf_1} & & \\
CX_1 \ar@{^{(}->}[r]^{i^{CX_1}} \ar[rrd]_{F_1} & C^2 X_1  \ar[rrd]_<<<<<<<{\Fm{2}{1}} & X_2
\ar@{^{(}->}[rrd]
\ar[d]^<<<{f_{2}} \ar@{^{(}->}[r]
&  C X_2 \ar@{^{(}->}[rrd]^{i^{CX_2}} \ar[d]^<<<<<{F_2} \\
& & X_3 \ar@{^{(}->}[rrd]_{i^{X_3}} \ar[r]^>>>>>>>{f_3} &
X_4 \ar[rrd]^>>>>>>>>>>{f_4} & CX_2 \ar[d] \ar@{^{(}->}[r] & C^2X_2 \ar[d]^{\Fm{2}{2}}  \\
&&  & &   CX_3 \ar[r]_{F_3} & X_5}
$$
\end{summary}

\begin{defn}\label{dtwocubes}
An $n$-th order cubical Toda system \w[]{\XF} in $\Dd$ is called \emph{strongly cofibrant} if the diagram \wref{eqtwocubes} is strongly cofibrant (see Definition \ref{rprojmodst}).
\end{defn}

\begin{remark}\label{rtwocubes}
Note that when \w{\XF} is strongly cofibrant, in particular, the three $n$-cubes \w[,]{\mathbb{C}^{(n)}X_1}
\w[,]{\Mm{n}{2}\XF} and \w{\xm{n}{n+3}\XF} are strongly cofibrant (this always holds for  \w{\mathbb{C}^{(n)}X_1} as long as \w{X_1} is cofibrant in $\Dd$).
\end{remark}

\begin{defn}\label{relation between systyms}
Two $n$-th order cubical Toda systems \w{\XF} and \w{\XG} are \emph{equivalent} (written \w[)]{\XF \toss \XG} if  the sequences
$$
\xymatrix{
\mathbb{C}^{(n)}X_1 \ar[rrr]^{\AAA{n}{1} \XF} &&&  \Mm{n}{2}\XF
 \ar[rrr]^{\BBB{n}{2}\XF} &&& \xm{n}{n+3}\XF
}
$$
and
$$
\xymatrix{
\mathbb{C}^{(n)}X'_1 \ar[rrr]^{\AAA{n}{1} \XG} &&&  \Mm{n}{2}\XG
 \ar[rrr]^{\BBB{n}{2}\XG} &&& \xm{n}{n+3}\XG
}
$$
are equivalent as diagrams under the relation \w{\tos} of definition \S \ref{rprojmodst}  .

As we shall show in Corollary \ref{ceqimpeqtoda} below, equivalent higher order cubical Toda systems have $\tos$-equivalent Toda
brackets.
\end{defn}

\begin{prop}
Every higher order cubical Toda system\w{\XF} is equivalent under the relation \w{\toss} to a strongly cofibrant one.
\end{prop}

This follows by induction from:

\begin{lemma}\label{lfiltr}
Given an indexing category $\In{}{}$ filtered as in \S \ref{rprojmodst} and a diagram \w[,]{\Ac : \In{}{} \to \Dd} any (strongly)
cofibrant replacement for its restriction \w{\Ac'} to \w{F\sb{i}} can be extended to a strongly cofibrant replacement for \w[.]{\Ac}
\end{lemma}

\begin{defn}\label{cone of cube}
The \emph{cone} of an $n$-cube \w{\Ac} is the $n$-cube \w{C\Ac} obtained by applying the cone functor to \w[,]{\Ac}
with the $n$-cube inclusion \w[.]{\mathfrak{i}^{\Ac} : \Ac \to C\Ac}
\end{defn}
\begin{defn}\label{diagnull}
Given an $n$-th order cubical Toda system\w[,]{\XF}
for each \w{1 \leq m \leq n} and \w[,]{1 \leq j \leq n-m+3 }  the corresponding \emph{diagrammatic nullhomotopy}
is the \wwb{m-1}cube map \w{\FFF{m}{j}=\FFF{m}{j}\XF : C\mathbb{C}^{(m-1)} X_j \to \xm{m-1}{j+m+1} }
is defined:
$$
\FFF{m}{j}(\vva) =
\begin{cases}
 \Fm{m}{j} & \text{if} ~  \vva=(1,\dots,1)  \\
C^{\deg(\vva)-\init(\vva)+1} \Fm{\init(\vva)}{j} &\text{otherwise}
\end{cases}
$$

Thus for \w{m=1} we have the map of $0$-cubes \w[,]{\FFF{1}{j} = F_j} while
\w{\FFF{3}{1} : C\mathbb{C}^{(2)} X_1 \to \xm{2}{5}} is given by:
$$
\xymatrix@R=20pt@C=40pt {  CX_1  \ar@{^{(}->}[d]_{} \ar@{^{(}->}[r]^{} \ar[rrd]^>>>>>>>>>>>>{Cf_1}
& C^2 X_1 \ar@{^{(}->}[d]^{} \ar[rrd]^{C^2 f_1} & & \\
C^2 X_1 \ar@{^{(}->}[r]^{} \ar[rrd]_{CF_1} & C^3 X_1  \ar[rrd]_<<<<<<<<<<{\Fm{3}{1}} &
CX_2 \ar[d]^{} \ar@{^{(}->}[r]_{C(i^{X_2})}
&  C^2 X_2 \ar[d]^{\Fm{2}{2}} \\
& & CX_3 \ar[r]_{F_3} & X_5  }
$$
\end{defn}
\begin{remark}
Note that \w{\FFF{m}{j}\XF} is indeed a diagrammatic nullhomotopy, because the following diagram of
$m$-cube maps commutes:
\mycccdiag[\label{eqdiagnull}]{  C\mathbb{C}^{(m)}X_j \ar@/^{1.5pc}/[rrrrrrd]^{\FFF{m+1}{j}\XF} \\
\mathbb{C}^{(m)}X_j \ar@{_{(}->}[u]^{i^{\mathbb{C}^{(m)}X_j}} \ar[rrr]^{\AAA{m}{j} \XF} &&&  \Mm{m}{j+1}\XF
 \ar[rrr]^{\BBB{m}{j+1}\XF} &&& \xm{m}{m+j+2}\XF }
\end{remark}
%
%
%

\supsect{\protect{\ref{GTb}}.C}{Cubical Toda bracket in terms of the homotopy cofiber of cubes}

We show how the cubical Toda bracket defined above can be described in terms of the homotopy cofibers of maps of cubes.

\begin{defn}\label{partial}
For each \w{1\leq j \leq n} and \w[,]{\varepsilon=0,1} we let \w{\pa{j}{\varepsilon}\In{n}{} } denote the full subcategory of \w{\In{n}{}}
(isomorphic to \w[)]{\In{n-1}{}} consisting of sequences $\vva$ with \w[.]{a\sb{j}=\varepsilon}
\end{defn}

\begin{defn}\label{pa}
For any $n$-cube \w{\Ac} and \w[,]{1 \leq j \leq n } we denote by \w{ \pa{j}{\varepsilon} \Ac} the restriction of  \w{\Ac} to
\w{\pa{j}{\varepsilon}\In{n}{}} \wb[,]{\varepsilon=0,1} with the obvious \wwb{n-1}cube map
\w[.]{\partial\sp{j}\Ac : \pa{j}{0} \Ac \to \pa{j}{1} \Ac}  By functoriality any $n$-cube map \w{\mathfrak{F} : \Ac \to \mathbb{B}}
can be thought of as a commuting diagram of \wwb{n-1}cube maps, denoted by \w[:]{\partial\sp{j}\mathfrak{F}}
\begin{myeq}\label{diagrampa}
\xymatrix@R=15pt {  \pa{j}{0}\Ac \ar[d]_{\partial\sp{j}\Ac} \ar[rr]^{\pa{j}{0}\mathfrak{F}}  &&
 \pa{j}{0}\mathbb{B} \ar[d]^{\partial\sp{j}\mathbb{B}}  \\
\pa{j}{1}\Ac \ar[rr]^{\pa{j}{1} \mathfrak{F}} && \pa{j}{1}\mathbb{B}
}
\end{myeq}

\end{defn}

\begin{lemma}\label{paC}
\w{\partial^{m} \mathbb{C}^{(m)}X = i^{\mathbb{C}^{(m-1)}X} }
\end{lemma}

\begin{proof}
See Definitions \ref{pa} and \ref{cone of cube}, exemplified in \wref[.]{eqc2x}
\end{proof}

\begin{defn}\label{Gen cof and hcof}
Given an $m$-cube \w{\Ac} in $\Dd$, its \emph{cofiber} \w{\cf{m}{\Ac}} is the colimit of
the diagram \w{\RRR \Ac:\Ln{m}{}\to\Dd} given by \w{\pa{m+1}{0}\RRR \Ac =\Ac} and \w{\RRR \Ac(\vva)=\ast} elsewhere.

The \emph{homotopy cofiber} of \w{\Ac} is the colimit \w{\hcof^{(m)}(\Ac)} of
the diagram \w{R\Ac:\Ln{m}{} \to\Dd} obtained from $\RRR$ by replacing the values $\ast$ with the corresponding cones.
The natural map \w{\hcof^{(n)}(\Ac) \to \cof^{(n)}(\Ac) } by \w[.]{\zeta_{\Ac}}

Thus if $\Ac$ is the following square:
\mybbbdiag[\label{eqhcsq}]{
X \ar[rr]^{h} \ar[d]_{f}  &&  Z \ar[d]^{g}    \\
Y \ar[rr]_{k} && W
}
then \w{\RRR \Ac} is the left hand diagram and \w{R\Ac} is the right hand diagram in
$$
\xymatrix@R=20pt {  X \ar[d]^{f} \ar[r]^{h} \ar[rrd] & Z \ar[d]_{g}
\ar[rrd] & & \\
Y \ar[r]^{k} \ar[rrd] & W   & \ast \ar[d] \ar[r]
 &  \ast  \\
& & \ast  &   } \hspace*{10mm}
\xymatrix@R=20pt {  X \ar[d]^{f} \ar[r]^{h}
\ar@{^{(}->}[rrd]^>>>>>>{i^X} & Z \ar[d]_{g}
\ar@{^{(}->}[rrd]^{i^Z} & & \\
Y \ar[r]^{k} \ar@{^{(}->}[rrd]^>>>>>>>>{i^Y} & W
& CX \ar[d]^{Cf} \ar[r]^{Ch}
 &  CZ  \\
& & CY  &   }
$$
\end{defn}

\begin{remark}\label{cofzero}
If \w{f: X \to Y} is a map in \w[,]{\mathcal{D}} then \w[]{\cof^{(1)}(f)} is the usual cofiber
of \w[,]{f} and similarly for the homotopy cofiber.
\end{remark}
\begin{example}\label{cof(CX)}
For any \w[,]{X\in\Dd} \w[]{\cof^{(m)}(\mathbb{C}^{(m)} X)=
 \ee{m}{X} } and \w[.]{\hcof^{(m)}(\mathbb{C}^{(m)} X)=L^m X}
(see Definition \ref{CC^mX})
\end{example}
\begin{lemma}\label{cof=cof}
Applying \w{\cof^{(n-1)}} to the diagram  \w{\partial^{j}\mathfrak{F}} of \wref{diagrampa} yields:
\myrrrdiag[\label{coffdiag}]{ \cof^{(n-1)}(\pa{j}{0}\Ac) \ar[d]_{\cof^{(n-1)}(\partial^{j}\Ac)}
\ar[rr]^{\cof^{(n-1)}(\pa{j}{0}\mathfrak{F})}  &&
 \cof^{(n-1)}(\pa{j}{0}\mathbb{B}) \ar[d]^{\cof^{(n-1)}(\partial^{j}\mathbb{B})}  \\
\cof^{(n-1)}(\pa{j}{1}\Ac) \ar[rr]^{\cof^{(n-1)}(\pa{j}{1} \mathfrak{F})}
 && \cof^{(n-1)}(\pa{j}{1}\mathbb{B})
}
and \w{\cof^{(n)}(\mathfrak{F})} is just \w{cof^{(1)}} applied to the horizontal $1$-cube map in \wref[.]{coffdiag}
\end{lemma}

\begin{remark}\label{hcof=hcof}
Similarly, for any strongly cofibrant $n$-cube \w{\Ac} the homotopy cofiber \w{\hcf{n}{\Ac}}  is the colimit of:
$$
\xymatrix@R=13pt { \hcf{n-1}{\pa{n}{0}\Ac} \ar[rrr]^{\hcf{n-1}{\partial^{n}\Ac}} \ar@{^{(}->}[d]  &&& \hcf{n-1}{\pa{n}{1}\Ac} \\
C \Ac(1,\dots,1,0)
}
$$

For example,  the homotopy cofiber of any strongly cofibrant square \wref{eqhcsq} is the colimit of
\w[,]{ CY \hookleftarrow \hcof(f) \rightarrow \hcof(g) } where \w[]{\hcof(f)\hra CY} is the induced map in
$$
\xymatrix@R=15pt @C=10pt{  X \ar@{^{(}->}[d]_{i^X} \ar@{^{(}->}[r]^{f}  & Y \ar@/^{1.5pc}/[ddr]^{i^Y} \ar[d]   \\
CX \ar@/_{1.5pc}/[drr]_{Cf} \ar[r]  & \hcof(f) \ar@{-->}[dr]^{}  \\
& & CY
}
$$
\end{remark}

\begin{prop}\label{gen hcolim=cof}
If \w{\Ac} is a strongly cofibrant $n$-cube, the natural map \w{\zeta_{\Ac} : \hcof^{(n)}(\Ac) \to \cof^{(n)}(\Ac)} is weak equivalence.
\end{prop}

\begin{proof}
By induction on \w[:]{n} we know it is true for \w[.]{n=1} If the statement holds
for\w[,]{n=m-1} let \w{\Ac} be a strongly cofibrant $m$-cube; by hypothesis
\w{\hcf{m-1}{\pa{m}{0}\Ac} \to \cf{m-1}{\pa{m}{0}\Ac}} and \w{\hcf{m-1}{\pa{m}{1}\Ac} \to \cf{m-1}{\pa{m}{1}\Ac}} are  weak equivalences,
so the following diagram:
$$
\xymatrix@R=15pt @C=10pt { \hcf{m-1}{\pa{m}{0}\Ac} \ar[rrrrd]^{\simeq} \ar[rrr]^{\hcf{m-1}{\partial^{m}\Ac}} \ar@{^{(}->}[d]  &&& \hcf{m-1}{\pa{m}{1}\Ac} \ar[rrrd]^{\simeq} \\
 C \Ac(1,\dots,1,0)  \ar[rrrrd] &&&&  \cf{m-1}{\pa{m}{0}\Ac} \ar[rr]_{\cf{m-1}{\partial^{m}\Ac}} \ar[d]  && \cf{m-1}{\pa{m}{1}\Ac} \\
&&&& \ast
}
$$
induces the required weak equivalence on pushouts (since the upper front map is a cofibration, as in Lemma \ref{cofibrant square induces cofibration}).
\end{proof}

\begin{corollary}\label{hcof(CX)}
For any cofibrant \w[,]{X\in\Dd} the natural map \w{\zeta_{\mathbb{C}^{(m)}X}: L^m X \to \ee{m}{X} } is a weak equivalence.
\end{corollary}

\begin{lemma}\label{Important2}
If \w{\Ac} is a strongly cofibrant $n$-cube, where \w{\Ac(\vva)\simeq \ast} for every \w[,]{\vva \neq (1,\dots,1)}
then the structure map \w{r^{\Ac}: \Ac(1,\dots,1) \to \cf{n}{\Ac}} is a weak equivalence.
\end{lemma}
\begin{remark}\label{hcoftoB}
Assume given an $n$-cube map \w{\mathfrak{F} : \Ac \to \Bc} such that, for every \w[,]{\vva \neq (1,\dots,1)} \w{\mathfrak{F}(\vva):\Ac(\vva)\to\Bc(\vva)=C\Ac(\vva)} is the cone inclusion. We can
extend \w{\mathfrak{F}} to a map \w{R\Ac\to \mathbb{Q}} of \ww{\Ln{n}{}}-diagrams, where \w{\pa{m+1}{0}\mathbb{Q}=\Bc} and all the new maps are identities. Taking colimits yields a map
\w[,]{\hcof'^{(n)}(\mathfrak{F}) : \hcof^{(n)}(\Ac) \to \mathbb{B}(1,\dots,1)} where \w{\mathbb{B}(1,\dots,1)} is a homotopy cofiber of the cube \w[.]{\Bc}

Observe also that the map \w{\hcof'^{(n)}(\mathfrak{F}) } is obtained from the maps:
\begin{itemize}
\item   \w{ R\Ac(\vec{r}_{k})=C\Ac(\vec{r}_{k})=\Bc(\vec{r}_{k}) \to \Bc(1,\dots,1) }for all \w{1 \leq k \leq n}
\item    \w{R\Ac(\vec{r}_{n+1})=\Ac(1,\dots,1)\stk{\mathfrak{F}(1,\dots,1)} \Bc(1,\dots,1) }
\end{itemize}
\end{remark}
\begin{prop}\label{T=AcircB}
Given an $n$-order cubical Toda system\w[,]{\XF} then:
$$
\Tm{n}{1}\XF =\hcof'^{(n)}(\BBB{n}{2}) \circ \hcf{n}{\AAA{n}{1}}
$$
\end{prop}
\begin{proof}
By Definition \ref{Gen defn 1 of 2} and Remark \ref{hcoftoB}, it suffices to show that
$$
\Fm{n-(n+1)+1}{n+1+1}\circ C^{n+1-(n+1)}(\Fm{n+1-1}{1})~=~ \BBB{n}{2}(1,\dots,1) \circ R\AAA{n}{1}(\vec{r}_{n+1})
$$
and
$$
\Fm{n-k+1}{k+1}\circ C^{n+1-k}(\Fm{k-1}{1})~=~ (\xm{n}{n+3}(\vec{r}_k) \to \xm{n}{n+3}(1,\dots,1)) \circ R\AAA{n}{1}(\vec{r}_k) \quad ,\quad \forall 1 \leq k \leq n~.
$$

For the first equation note that \w{\Fm{0}{n+2}=f_{n+2}=\BBB{n}{2}(1,\dots,1)}
and \w[]{C^0\Fm{n}{1}=\Fm{n}{1}= \AAA{n}{1}(\vec{r}_{n+1})=R\AAA{n}{1}(\vec{r}_{n+1})} (see Definitions \ref{AAA} and \ref{BBB}).

For the second, note that for \w{1 \leq k \leq n} by Definition \ref{xm} we have
\w{\Fm{n-k+1}{k+1} =\xm{n}{n+3}(\vec{r}_{k}^{(n)}) \to \xm{n}{n+3}(1,\dots,1)} and
\w[.]{
C^{n-k+1}F_1^{(k-1)}=C C^{n-k} \Fm{k-1}{1}= C (\AAA{n}{1}(\vec{r}_k^{(n)}))=R\AAA{n}{1}(\vec{r}_k^{(n+1)})}

\end{proof}
\begin{corollary}\label{ceqimpeqtoda}
If \w{\XF \toss \XG } then \w[.]{\Tm{n}{1}\XF \tos\Tm{n}{1}\XG}
\end{corollary}

%
%
\sect{Passage from the cubical to the recursive definitions}
\label{DdTB}
We now give a diagrammatic description of recursive Toda systems, and use this to describe the recursive Toda bracket in terms of
(homotopy) cofibers of cubes. This will allow us to relate the two definitions of general Toda brackets.

\supsect{\protect{\ref{DdTB}}.A}{Diagrammatic interpretation for the recursive system}

We now show that every recursive Toda system is equivalent
to one with a rectified final segment (as in \S \ref{OrdToda}.C), and use this to provide a diagrammatic description for
recursive Toda systems.
\begin{defn}
Given an $n$-th order recursive Toda system \w[,]{\XFn} we define \w{\Rec \XFn} to be the following sequence:
$$
\xymatrix@C=35pt{  X_1 \ar[r]^{f_1} & X_2 \ar[r]^{f_2} & X_3
\ar[r]^{ \bnn{1}{2} \circ r^{f_2}} &
 \cof(\Fn{1}{2}) \ar[r]^>>>>{\bnn{2}{2} \circ r^{\bnn{1}{2}}}
&\dots \cof(\Fn{n-1}{2}) \ar[r]^>>>>{\bnn{n}{2} \circ r^{\bnn{n-1}{2}} }  & \cof(\Fn{n}{2})
}
$$
with nullhomotopies \w{\Gn{m}{1}:=\Fn{m}{1}} and
\w{\Gn{m}{j}:=\ast} for all \w[.]{j \geq 2}
\end{defn}
\begin{lemma}\label{lrecred}
The system \w{\Rec\XFn} is an $n$-th order recursive Toda system (in the extended sense of \S \ref{rredchcx}),
and \w[.]{\Tn{n}{1} (\Rec \XFn) = \Tn{n}{1} \XFn }
\end{lemma}

Note that even though \w{\Rec\XFn} is not a recursive higher Toda system in the strict sense of \S \ref{dredchcx}, we can deduce from the
Lemma that the corresponding Toda bracket \w{\Tn{n}{1} (\Rec \XFn)} is in fact homotopy meaningful (despite Remark \ref{rredchcx}).
\begin{lemma}\label{lnullcomp}
Given an $m$-th order recursive Toda system \w[,]{\XFn} for any \w[,]{m \geq 1}
\w{\Fn{m+1}{1}} is a nullhomotopy for the composite \w{\bnn{m}{2} \circ \ann{m}{1}}
if and only if
\mydiagram[\label{eqnullcomp}]{
C\ee{m-1}{X_1} \ar[rrr]^{i^{\ee{m}{X_1}}\circ\, r^{i^{\ee{m-1}{X_1}}}} \ar[d]_{\Fn{m}{1}}
&&& C\ee{m}{X_1}  \ar[d]^{\Fn{m+1}{1}}\\
\cof(\Fn{m-1}{2}) \ar[rrr]_{\bnn{m}{2} \circ\, r^{\bnn{m-1}{2}}} &&& \cof(\Fn{m}{2})
}
commutes.
\end{lemma}

\begin{proof}
We have the following diagram of horizontal cofibration sequences:
$$
\xymatrix@R=20pt {  \ee{m-1}{X_1} \ar@{^{(}->}[rr]^{i^{\ee{m-1}{X_1}}} \ar@{^{(}->}[d]_{\ann{m-1}{1}}
&& C \ee{m-1}{X_1} \ar@{^{(}->}[d]^{\Fn{m}{1}}  \ar[rr]^{ r^{i^{\ee{m-1}{X_1}}}} &&
\ee{m}{X_1} \ar[d]^{\ann{m}{1}} \\
\cof(\bnn{m-2}{2})  \ar@{^{(}->}[rr]^{\bnn{m-1}{2}}  && \cof(\Fn{m-1}{2})
\ar[rr]^{r^{\bnn{m-1}{2}}} && \cof(\bnn{m-1}{2})
}
$$
The right hand square and that \w{\Fn{m+1}{1}} is a nullhomotopy for \w{\bnn{m}{2} \circ \ann{m}{1}} yields:

$$
\bnn{m}{2} \circ r^{\bnn{m-1}{2}} \circ \Fn{m}{1}~=~\bnn{m}{2} \circ \ann{m}{1} \circ r^{i^{\ee{m-1}{X_1}}}~=~
\Fn{m+1}{1} \circ i^{\ee{m}{X_1}} \circ r^{i^{\ee{m-1}{X_1}}}~
$$
as in \wref[.]{eqnullcomp}

Conversely, if \wref{eqnullcomp} commutes, again using the same right hand square, we get
\w[.]{\bnn{m}{2} \circ \ann{m}{1} \circ r^{i^{\ee{m-1}{X_1}}}= \bnn{m}{2} \circ r^{\bnn{m-1}{2}} \circ \Fn{m}{1}
=\Fn{m+1}{1} \circ i^{\ee{m}{X_1}} \circ r^{i^{\ee{m-1}{X_1}}}}
Since \w{r^{i^{\ee{m-1}{X_1}}}} is epic,
\w[.]{\bnn{m}{2} \circ \ann{m}{1} = \Fn{m+1}{1} \circ i^{\ee{m}{X_1}}}
\end{proof}
\begin{defn}\label{CCnX}
For any \w[,]{X\in\Dd} the $n$-cube \w{\CCn{n}{X}} in $\Dd$ is defined:
$$
\CCn{n}{X}(\vva) :=
\begin{cases}
C\ee{-1}{X}:=X   &   \text{if} ~ \vva=(0,\dots,0) \\
 C\ee{\deg(\vva)-1}{X} & \text{if}~ \deg(\vva)=\init(\vva)  \\
 \ast  & \text{otherwise}
\end{cases}
$$
The only nonzero maps are \w[,]{i^{\ee{m}{X}}\circ\, r^{i^{\ee{m-1}{X}}}:C\ee{m-1}{X}\to C\ee{m}{X}} where
\w{m=\deg(\vva)} (see Definition \ref{dcubicals}).

Thus the $3$-cube \w{\CCn{3}{X}} is
$$
\xymatrix@R=20pt@C=40pt {  X  \ar@{^{(}->}[d]^{i^{X}}  \ar[r]
\ar[rrd] & \ast \ar[d] \ar[rrd] & & \\
CX \ar[r]^{i^{\eex}\circ r^{i^{X}}} \ar[rrd] & C\eex
\ar[rrd]^>>>>>>>>{{i^{\ee{2}{X}}\circ r^{i^{\eex}}}} & \ast \ar[r]
\ar[d]  &  \ast \ar[d] \\
& & \ast \ar[r] & C \ee{2}{X}  }
$$
\end{defn}

\begin{defn}\label{Mn}
Given an $n$-th order recursive Toda system \w[,]{\XFn} we define an $n$-cube
\w{\Mn{n}{2}=\Mn{n}{2}\XFn} in $\Dd$ by:
$$
\Mn{n}{2}(\vva):=\begin{cases}
X_2 & \text{if} \ \vva=(0,\dots,0)\\
X_3 & \text{if}\ \vva=(1,0,\dots,0)\\
\cof(\Fn{k-1}{2}) & \text{if}\ \deg{(\vva)}=\init{(\vva)}=k \geq 2\\
\ast & \text{otherwise}
\end{cases}
$$
Again, the only nonzero maps are \w[,]{\bnn{k}{2} \circ r^{\bnn{k-1}{2}}:\Mn{n}{2}(\vva)\to \Mn{n}{2}(\vvb)}
when \w{\deg(\vva)=\init(\vva)=k} and \w[,]{\deg(\vvb)=\init(\vvb)=k+1}
with \w{\bnn{0}{2}:=f_2} and \w[.]{r^{\bnn{-1}{2}}:=\Id\sb{X_{2}}}

Thus the $3$-cube \w{\Mn{3}{2}\XFn} is the following
$$
\xymatrix@R =20pt@C=30pt {
X_2 \ar[rrd]
\ar[d]_{f_{2}} \ar[r]
&  \ast \ar[rrd] \ar[d]^<<<<{} \\
 X_3 \ar[rrd] \ar[r]^{\bnn{1}{2} \circ r^{f_2}}
& \cof(\Fn{1}{2}) \ar[rrd]^{\bnn{2}{2} \circ r^{\bnn{1}{2}}} & \ast \ar[d] \ar[r] & \ast \ar[d] \\
&&   \ast \ar[r] & \cof(\Fn{2}{2})}
$$
\end{defn}
\begin{defn}\label{AAAn}
The $n$-cube map \w{\AAAn{n}{1}=\AAAn{n}{1}\XF : \CCn{n}{X_1} \to \Mn{n}{2} } is defined:
$$
\AAAn{n}{1}(\vva) =
\begin{cases}
\Fn{k}{1} & \text{if} ~ \deg(\vva)=\init(\vva)=k   \\
\ast & \text{otherwise}
\end{cases}
$$
where \w[.]{\Fn{0}{1}=f_1}
Lemma \ref{lnullcomp} implies that \w{\AAAn{n}{1}\XF} is well defined.

Thus the $2$-cube map \w{\AAAn{2}{1} \XFn:\CCn{2}{X_1}\to\Mn{2}{2}} is the diagonal map in:
\mydiagram[\label{eqtwoone}]{  X_1  \ar@{^{(}->}[d] \ar[r] \ar[rrd]^>>>>>>>>>>{\Fn{0}{1}=f_1} & \ast \ar@{^{(}->}[d]
\ar[rrd]^{} & & \\
CX_1 \ar[r]^{} \ar[rrd]_{\Fn{1}{1}} & C \ee{}{X_1}  \ar[rrd]_<<<<<<<{\Fn{2}{1}} & X_2
\ar[d]^<<<{f_{2}} \ar[r]
&  \ast \ar[d]^<<<<{} \\
& & X_3  \ar[r]_{\bnn{1}{2} \circ r^{f_2}}
& \cof(\Fn{1}{2}) }
\end{defn}
\begin{defn}\label{Fin}
Given \w[,]{X \in \Dd} we define \w{\Fin{n}{X} } to be the $n$-cube with \w{X} in the \wwb{1,\dots,1}slot and $\ast$
elsewhere.  This is functorial in $X$.

Now given an $n$-th order recursive Toda system \w[,]{\XFn} we define the $n$-cube
\w{\xn{n}{n+3}=\xn{n}{n+3}\XFn} in $\Dd$ to be \w[.]{\Fin{n}{\cof(\Fn{n}{2})}}
\end{defn}

\begin{defn}\label{BBBn}
Given an $n$-th order recursive Toda system \w[]{\XFn} in $\Dd$, the $n$-cube map \w{\BBBn{n}{2}=\BBBn{n}{2}\XFn : \Mn{n}{2} \to  \xn{n}{n+3}}  is
defined by \w[.]{\BBBn{n}{2}(1,\dotsc,1) =\bnn{n}{2} \circ r^{\bnn{n-1}{2}}}
\end{defn}
\begin{summary}\label{nstwocubes}
Just as the data of a higher order cubical Toda system is encoded in the sequence of $n$-cube maps \wref[,]{stwocubes} so
for any $n$-th order recursive Toda system \w[,]{\XFn} the data for \w{\Rec\XFn} is encoded by:
\mydiagram[\label{neqtwocubes}]{
\CCn{n}{X_1} \ar[rrr]^{\AAAn{n}{1} \XFn} &&&  \Mn{n}{2}\XFn
 \ar[rrr]^{\BBBn{n}{2}\XFn} &&& \xn{n}{n+3}\XFn
}

Thus for a second-order recursive Toda system \w{\XFn^{(2)}} of length $4$, \w{\Rec\XFn} is encoded by:
$$
\xymatrix@R =25pt@C=30pt {  X_1  \ar@{^{(}->}[d] \ar[r] \ar[rrd]^>>>>>>>>>>{f_1} & \ast \ar[d]
\ar[rrd]^{} & & \\
CX_1 \ar[r]^{} \ar[rrd]_{\Fn{1}{1}} & C \ee{}{X_1}  \ar[rrd]_<<<<<<<{\Fn{2}{1}} & X_2 \ar[rrd]
\ar[d]^<<<{f_{2}} \ar[r]
&  \ast \ar[rrd] \ar[d]^<<<<{} \\
& & X_3 \ar[rrd] \ar[r]^>>>>>>>{\bnn{1}{2} \circ r^{f_2}}
& \cof(\Fn{1}{2}) \ar[rrd]^{\bnn{2}{2} \circ r^{\bnn{1}{2}}} & \ast \ar[d] \ar[r] & \ast \ar[d] \\
&&  & &   \ast \ar[r] & \cof(\Fn{2}{2})}
$$
\end{summary}
\begin{defn}\label{dtosr}
Two $n$-th order recursive Toda systems \w{\XFn} and \w{\XGn} are \emph{equivalent}
(written \w[)]{\XFn \tosr \XGn} if  the sequences
\mydiagram[\label{eqtosra}]{
\CCn{n}{X_1} \ar[rrr]^{\AAAn{n}{1} \XFn} &&&  \Mn{n}{2}\XFn
 \ar[rrr]^{\BBBn{n}{2}\XFn} &&& \xn{n}{n+3}\XFn
}
and
\mydiagram[\label{eqtosrb}]{
\CCn{n}{X'_1} \ar[rrr]^{\AAAn{n}{1} \XGn} &&&  \Mn{n}{2}\XGn
 \ar[rrr]^{\BBBn{n}{2}\XGn} &&& \xn{n}{n+3}\XGn
}
are equivalent under the relation \w{\tos} of \S \ref{rprojmodst}.

We will show later that \ww{\tosr}-equivalent higher order recursive Toda systems
have $\tos$-equivalent recursive Toda brackets.
\end{defn}
%
%
%
%
%

\begin{lemma}\label{paCCnX}
For any \w[,]{X \in \Dd}  \w{\pa{n}{0}\CCn{n}{X}=\CCn{n-1}{X}} and \w{\pa{n}{1}\CCn{n}{X}=\Fin{n-1}{C \ee{n-1}{X}}}
and
\w{\cof^{(n-1)} (\partial^{n} \CCn{n}{X})= i^{\ee{n-1}{X}} : \ee{n-1}{X} \to C\ee{n-1}{X} }
(see Definitions  \ref{Gen cof and hcof}  and \ref{pa}).
\end{lemma}
\begin{lemma}\label{npa}
Given an $n$-th order recursive Toda system \w[]{\XFn} in $\Dd$, then
\begin{enumerate}
\renewcommand{\labelenumi}{(\alph{enumi})~}
\item \w{\pa{n}{0}\AAAn{n}{1}\XFn = \AAAn{n-1}{1}\XFn}
(Definitions \ref{pa} and \ref{AAAn})
\item \w{\partial^n \Mn{n}{2}\XFn = \BBBn{n-1}{2}\XFn}
(Definitions \ref{Mn} and \ref{BBBn})
\item \w{ \pa{n}{1}\AAAn{n}{1}  =
\Fin{n-1}{\Fn{n}{1}}: \Fin{n-1}{C\ee{n-1}{X_1}} \to \Fin{n-1}{\cof(\Fn{n-1}{2})}
}
(Definition \ref{Fin} )
\end{enumerate}
\end{lemma}
\begin{prop}\label{nToda=cofb circ cofa}
For any $n$-th order recursive Toda system \w[]{\XFn} in $\Dd$:
$$
\Tn{n}{1}\XFn= \cof^{(n)}(\BBBn{n}{2})
\circ \cf{n}{\AAAn{n}{1}}
$$
see Definitions \wref[,]{AAAn} \wref{BBBn}  and \wref[.]{Gen defn 2 of 2}
\end{prop}
\begin{proof}
We have to show that:
\begin{myeq}\label{eqindstep}
\cof^{(n)}(\AAAn{n}{1}\XFn) = \ann{n}{1}\XFn \ \ \text{and} \ \ \cof^{(n)}(\BBBn{n}{2}\XFn)=\bnn{n}{2}\XFn
\end{myeq}
By Lemma \ref{cof=cof}, applying \w{\cof^{(n)}} to the map $n$-cubes \w{\AAA{n}{1} : \CCn{n}{X_1} \to \Mn{n}{2} }
is equivalent to applying \w{\cof^{(1)}} to the following diagram (i.e., horizontal map)
of vertical $1$-cubes:
\mydiagram[\label{aaabsquare}]{
\cof^{(n-1)}(\pa{n}{0} \CCn{n}{X_1})
\ar[d]_{\cof^{(n-1)}(\partial^{n}\CCn{n}{X_1})}
 \ar[rrr]^{\cof^{(n-1)}(\pa{n}{0}\AAAn{n}{1})}  &&&
 \cof^{(n-1)}(\pa{n}{0}\Mn{n}{2}) \ar[d]^{\cof^{(n-1)}(\partial^{n}\Mn{n}{2})}  \\
\cof^{(n-1)}(\pa{n}{1}\CCn{n}{X_1}) \ar[rrr]^{\cof^{(n-1)}(\pa{n}{1}\AAAn{n}{1})} &&& \cof^{(n-1)}(\pa{n}{1}\Mn{n}{2})
}
By Lemmas \ref{paCCnX} and  \ref{npa}, diagram \wref{aaabsquare} equals:
\mydiagram[\label{aaaabsquare}]{
 \ee{n-1}{X_1}
\ar@{^{(}->}[d]_{i^{\ee{n-1}{X_1}}}
 \ar[rrr]^{\cof^{(n-1)}(\AAAn{n-1}{1})}  &&&
 \cof^{(n-1)}(\Mn{n-1}{2}) \ar[d]^{\cof^{(n-1)}(\BBBn{n-1}{2})}  \\
C\ee{n-1}{X_1} \ar[rrr]^{\Fn{n}{1}} &&&
\cof(\Fn{n-1}{2})
}
so the result follows by induction on \wref[.]{eqindstep} The result for \w{\BBBn{n}{2}} is shown analogously.
\end{proof}
\begin{corollary}\label{Tn=Tn}
Given two higher order recursive Toda systems \w[,]{\XFn \tosr \XGn} we have \w[.]{\Tn{n}{1}\XFn \tos \Tn{n}{1}\XGn}
\end{corollary}

Here we use the fact that a zigzag of weak equivalences between \w{\CCn{n}{X_1}} and \w{\CCn{n}{X'_1}} induces a
zigzag of weak equivalences on their (strict) cofibers from \w[] {\ee{n}{X_1}} to \w[]{\ee{n}{X'_1}} (see Definition \ref{CCnX}).

%
%
\supsect{\protect{\ref{DdTB}}.B}{Passage from the cubical to the recursive definition}

We now prove the first direction in showing that cubical and recursive definitions of the Toda brackets agree,
using the diagrammatic description to show how a cubical Toda bracket is $\tos$-equivalent to the recursive version.

\begin{defn}\label{dva}
Given an $n$-cube \w{\Ac} in $\Dd$, we define an $n$-cube \w{\VV{\Ac}} with a natural transformation \w{\vartha{\Ac}:\Ac\to\VV{\Ac}}
by induction on \w{n\geq 1} as follows:

For \w[,]{n=1} \w{\vartha{\Ac}} is \w[.]{\Id:\Ac\to\Ac}
For \w{n\geq 2} we set \w{\pa{n}{0}\VV{\Ac}:=\VV{\pa{n}{0}\Ac}} and define the \wwb{n-1}cube \w{\pa{n}{1}\VV{\Ac}:=\Fin{n-1}{\cf{n-1}{\pa{n}{1}\Ac}}}
(see Definitions \ref{pa} and \ref{Fin}),
where the map \w{\VV{\Ac}(1,\dots,1,0) \to \VV{\Ac}(1,\dots,1,1)} is
\w{ \cf{n-1}{\partial^{n}\Ac} \circ r^{\cf{n-2}{\partial^{n-1}\pa{n}{0}\Ac}}  }
and \w{\vartha{\Ac}} in the final vertex is the structure map \w{r\sp{\pa{n}{1}\Ac}} (see Definition \ref{Gen cof and hcof} and
Remark \ref{cofzero}).
\end{defn}
\begin{example}
If \w{\Ac} is the back square in the following, then \w{\vartha{\Ac}:\Ac\to\VV{\Ac}} is:
$$
\xymatrix@R=20pt {  X \ar[d]^{f} \ar[r]^{h} \ar@{=}[rrd] & Z \ar[d]_{g}
\ar[rrd] & & \\
Y \ar[r]^{k} \ar@{=}[rrd] & V \ar[rrd]^>>>>>>>{r^g}  & X \ar[d]_>>>{f} \ar[r]
 &  \ast \ar[d] \\
& & Y \ar[r]_{r^g \circ k} &  \cof(g) }
$$
(Observe that \w[).]{r^g \circ k =(\cof(f) \to \cof(g))\circ r^f }
\end{example}
\begin{example}\label{egva}
For any \w[,]{X\in\Dd} we have \w{\VV{\mathbb{C}^{(n)}X}= \CCn{n}{X}} (See definitions \ref{CC^mX} and \ref{CCnX}).
\end{example}
\begin{lemma}\label{lcfw=cf}
Given an $n$-cube \w[,]{\Ac} then \w[.]{\cf{n}{\vartha{\Ac}}= \Id_{\cf{n}{\Ac}}} In particular
\w[.]{\cf{n}{\VV{\Ac}}=\cf{n}{\Ac}}
\end{lemma}
\begin{lemma}\label{lnatwe}
If \w{\Ac} is a strongly cofibrant $n$-cube in $\Dd$ and \w{\Ac(\vva)\simeq\ast} whenever \w{\deg(\vva)\neq\init(\vva)} (Definition \ref{dcubicals}),
then \w{\vartha{\Ac}} is a weak equivalence.
\end{lemma}
%
%
\begin{lemma}\label{Important}
Given a map of $n$-cubes \w[,]{\mathbb{F} :\Ac \to \Bc} we have a commutative diagram:
$$
\xymatrix@R=15pt @C=55pt {
\Ac \ar[rr]^{\mathbb{F}} \ar[d]_>>>>{\vartha{\Ac}} && \Bc \ar[d]^>>>{(r^{\Bc})_{\ast}} \\
W(\Ac) \ar[rr]_<<<<<<<<<<<<<<<<<<{(\cf{n}{\mathbb{F}} \circ r^{\cf{n-1}{\partial^{n}\Ac}} )_{\ast}} && \Fin{n}{\cf{n}{\Bc}}
}
$$
(in the notation of Lemma \ref{Important2}).
\end{lemma}
\begin{proof}
If we think of the $n$-cube map \w{\mathbb{F}} as an $(n+1)$-cube, then the bottom map is \w[,]{\VV{\mathbb{F}}}
and the vertical map is \w[.]{\vartha{\mathbb{F}}}
\end{proof}

\begin{lemma}\label{paM}
Assume given a higher order cubical Toda system \w{\XF} in $\Dd$.
\begin{enumerate}
\renewcommand{\labelenumi}{(\alph{enumi})~}
\item \w{\partial\sp{1} \Mm{m}{j}\XF = \AAA{m-1}{j}\XF   } (Definitions \ref{pa}, \ref{AAA} and \ref{Mm})
\item \w{\partial\sp{m}\Mm{m}{j}\XF =  \BBB{m-1}{j}\XF } (Definition \ref{BBB}).
\item \w{\pa{m}{0}\AAA{m}{j}\XF = \AAA{m-1}{j} \XF }
and \w[.]{\pa{m}{1}\AAA{m}{j}\XF = \FFF{m}{j}\XF }
\item \w{\pa{1}{0} \BBB{m}{j}\XF =  \mathfrak{i}^{\mathbb{C}^{m-1}X_j}  }
and   \w[.]{\pa{1}{1} \BBB{m}{j}\XF =  \BBB{m-1}{j+1}\XF  }
\item \w{\partial\sp{1}\xm{m}{j+m+1}\XF = \FFF{m}{j}\XF} (Definitions \ref{xm} and \ref{diagnull}).
\end{enumerate}
\end{lemma}

\begin{defn}\label{indchcx}
Given an $n$-th order cubical Toda system\w{\XF} in $\Dd$ of length \w[,]{n+2}  for each \w{1 \leq m \leq n}
 and \w[,]{1 \leq j \leq n-m+2 } we write \w{\Fn{m}{j}:=\cof^{(m-1)}{(\FFF{m}{j}\XF ) }} (see \S \ref{diagnull}), and denote this system by \w[.]{V\XF:=\XFn}
\end{defn}

\begin{prop}\label{prop1}
Given a strongly cofibrant $n$-th order cubical Toda system\w{\XF} in $\Dd$ of length \w[,]{n+2} \w{V\XF} constitutes an $n$-th order recursive 
Toda system  (see Definition \ref{dredchcx}), satisfying
$$
\cof^{(m)}(\AAA{m}{j}\XF) = \ann{m}{j}(V\XF)\hspace{5mm}\text{and} \ \ \cof^{(m)}(\BBB{m}{j}\XF)=\bnn{m}{j}(V\XF)
$$
as in Definition \ref{Gen defn 2 of 2}.
\end{prop}

\begin{proof}
Because \w{\XF} is strongly cofibrant, by Lemma \ref{paM} and Remark \ref{rtwocubes} all the squares appearing in the construction of \w{\ann{m}{j}(V\XF)} and \w{\bnn{m}{j}(V\XF)} in \S \ref{Gen defn 2 of 2} are strongly cofibrant.

We prove the Proposition by induction on $m$:

If  \w[,]{m=1} the cofiber of the $1$-cube map \w[]{\AAA{1}{j}:(i^{X_{j}})\to(f_{j+1})} is \w[,]{\widetilde{\alpha}(f_j,f_{j+1},F_j)} and the cofiber of the $1$-cube
\w[]{\BBB{1}{j}:(f_{j})\to(F_{j})}  is \w[,]{\widetilde{\beta}(f_j,f_{j+1},F_j)} as in \wref[.]{eqalphat}
Here \w{\FFF{1}{j}} is  \w[,]{F_j} so \w{ {\widetilde{F}}^{(1)}_j ={\cof}^{(0)}(F_j)=F_j } and thus:
$$
\cof^{(1)}(\AAA{1}{j}) = \widetilde{\alpha}(f_j,f_{j+1}, {\widetilde{F}}^{(1)}_j )= \ann{1}{j}\hspace{4mm}\text{and}\ \
\cof^{(1)}(\BBB{1}{j}) =\widetilde{\beta}(f_j,f_{j+1}, {\widetilde{F}}^{(1)}_j )= \bnn{1}{j}~.
$$

Now suppose the statement is true for \w[.]{m-1}
By Proposition \ref{cof=cof}, \w{\cof^{(m)}} of \w{\AAA{m}{j} : \mathbb{C}^{(m)}X_j \to \Mm{m}{j+1} }
is the result of applying \w{\cof^{(1)}} to the following horizontal map of vertical $1$-cubes:
\mydiagram[\label{absquare}]{
\cof^{(m-1)}(\pa{m}{0} \mathbb{C}^{(m)}X_j) \ar[d]_{\cof^{(m-1)}(\partial^{m}\mathbb{C}^{(m)}X_j)}
 \ar[rr]^{\cof^{(m-1)}(\pa{m}{0}\AAA{m}{j})}  &&
 \cof^{(m-1)}(\pa{m}{0}\Mm{m}{j+1}) \ar[d]^{\cof^{(m-1)}(\partial^{m}\Mm{m}{j+1})}  \\
\cof^{(m-1)}(\pa{m}{1}\mathbb{C}^{(m)}X_j) \ar[rr]^{\cof^{(m-1)}(\pa{m}{1}\AAA{m}{j})} && \cof^{(m-1)}(\pa{m}{1}\Mm{m}{j+1})
}
By Lemmas \ref{paC} and \ref{paM}, diagram \wref{absquare} may be rewritten in the form:
\mydiagram[\label{abfmsquare}]{
\cof^{(m-1)}(\mathbb{C}^{(m-1)}X_j) \ar[d]_{\cof^{(m-1)}(\mathfrak{i}^{\mathbb{C}^{(m-1)}X_j } )}
 \ar[rr]^{\cof^{(m-1)}(\AAA{m-1}{j})}  &&
 \cof^{(m-1)}(\Mm{m-1}{j+1}) \ar[d]^{\cof^{(m-1)}(\BBB{m-1}{j+1})}  \\
\cof^{(m-1)}(C\mathbb{C}^{(m-1)}X_j) \ar[rr]^{\cof^{(m-1)}(\FFF{m}{j})} &&
\cof^{(m-1)}(\xm{m-1}{j+m+1})
}
By Example \ref{cof(CX)} and the induction hypothesis we get:
$$
\xymatrix@R=25pt {  \widetilde{\Sigma}^{(m-1)}X_j \ar[d]_{i^{\widetilde{\Sigma}^{(m-1)}X_j}}
 \ar[rr]^{\ann{m-1}{j}}  &&
 \cof^{(m-1)}(\bnn{m-2}{j+1}) \ar[d]^{\bnn{m-1}{j+1}}  \\
C\widetilde{\Sigma}^{(m-1)}X_j  \ar[rr]^{{\widetilde{F}}^{(m)}_j} && \cof^{(m-1)}({\widetilde{F}}^{(m-1)}_{j+1})  \\
}
$$
(where the right vertical map is obtained by commuting cones with cofibers).

Thinking of this as a horizontal map of vertical $1$-cubes (as above), we may identify it
with \w{\ann{m}{j}\XFn} (see \wref[).]{eqdefalpha}

A similar argument shows that \w[.]{\cof^{(m)}(\BBB{m}{j}\XF)=\bnn{m}{j}\XFn}
\end{proof}
\begin{corollary}
Under the assumptions of Proposition \ref{prop1} we have
$$
\Tn{n}{1}(V\XF)~=~
\cof^{(n)}(\BBB{n}{2}\XF) \circ \cof^{(n)}(\AAA{n}{1}\XF)
$$
\end{corollary}
\begin{prop}\label{XF to RecVXF}
Given a strongly cofibrant $n$-th order cubical Toda system \w{\XF} in $\Dd$, we have
a natural commuting diagram
$$
\xymatrix@R=25pt {
\mathbb{C}^{(n)}X_1 \ar[d]_{\vartha{\mathbb{C}^{(n)}X_1}}^{\simeq}
\ar[rrr]^{\AAA{n}{1} \XF} &&& \Mm{n}{2}\XF
\ar[d]^{\vartha{\Mm{n}{2}}}_{\simeq} \ar[rrr]^{\BBB{n}{2}\XF} &&& \xm{n}{n+3}\XF
\ar[d]_{\simeq}^{(r^{\xm{n}{n+3}})_\ast} \\
\CCn{n}{X_1} \ar[rrr]^>>>>>>>>>>>>>{\AAAn{n}{1} (V\XF)} &&&  \Mn{n}{2}(V\XF)
\ar[rrr]^{\BBBn{n}{2}(V\XF)} &&& \xn{n}{n+3}(V\XF)
}
$$
with vertical weak equivalences.
\end{prop}
\begin{proof}
Note that the cube map \w[]{\VV{\AAA{n}{1} \XF}: \VV{\mathbb{C}^{(n)} X_1} \to \VV{\Mm{n}{2}\XF}} is equal to \w[,]{\AAAn{n}{1} (V\XF)} so the left square commutes by naturality of $\vartha{}$ (Definition \ref{dva}),
and the two left vertical maps are weak equivalences by Lemma \ref{lnatwe}.
In addition \w{\BBBn{n}{2}(V\XF) =  (\bnn{n}{2}(V\XF) \circ r^{\bnn{n-1}{2}(V\XF)})_{\ast}} by Definition \ref{BBBn}, and it is equal to \w{(\cf{n}{\BBB{n}{2}\XF}\circ r^{\cf{n-1}{\BBB{n-1}{2}\XF}})_{\ast} } by
Proposition \ref{prop1}. By Lemma \ref{paM}(b)
we can replace \w{\BBB{n-1}{2}\XF} by \w[,]{\partial\sp{n}\Mm{n}{2}\XF} so
the right square commutes according to Lemma \ref{Important}, and by Lemma \ref{Important2} the right vertical map
is a weak equivalence.
\end{proof}

\begin{corollary}
If \w{\XF} and \w{\XG} are strongly cofibrant $n$-th order cubical Toda systems in $\Dd$ and
\w[,]{\XF \toss \XG} then \w[.]{V\XF \tosr V\XG}
\end{corollary}

We may summarize the results of this section in the following generalization of Proposition \ref{def 1 = def 2 of 2}:

\begin{thm}\label{def1=def2}
Given a strongly cofibrant $n$-th order cubical Toda system\w{\XF} in $\Dd$, we have a commutative diagram:
$$
\xymatrix@R=18pt {
L^n X_1  \ar[rrr]^{\Tm{n}{1}\XF}
 \ar[d]_{\zeta_{\mathbb{C}^{n}X_1}}^{\simeq}  &&& X_{n+3}
\ar[d]^>>>>{r^{\xm{n}{n+3}}}_{\simeq} \\
{\widetilde{\Sigma}}^{n} X_1
 \ar[rrr]_{\Tn{n}{1}(V\XF) }
&&& \cf{n}{\xm{n}{n+3}}
}
$$
\end{thm}

See Definitions \ref{Gen defn 1 of 2}, \ref{dredchcx}, and \ref{indchcx}.

\begin{proof}
Applying  Proposition \ref{gen hcolim=cof} and Lemma \ref{Important2} to \wref{eqtwocubes} yields the commuting diagram:
\myssdg[\label{ababab}]{
\hcf{n}{\mathbb{C}^{(n)}X_1} \ar[d]_{\zeta_{\mathbb{C}^{(n)}X_1}}^{\simeq}
\ar[rrr]^{\hcf{n}{\AAA{n}{1} }} &&& \hcf{n}{\Mm{n}{2}}
\ar[d]^{\zeta_{\Mm{n}{2}}}_{\simeq} \ar[rrr]^{\hcof'^{(n)}(\BBB{n}{2})} &&& X_{n+3}
\ar[d]_{\simeq}^{r^{\xm{n}{n+3}}} \\
\cf{n}{\mathbb{C}^{(n)}X_1} \ar[rrr]^>>>>>>>>>>>>>{\cf{n}{\AAA{n}{1} }} &&&  \cf{n}{\Mm{n}{2}}
\ar[rrr]^{\cf{n}{\BBB{n}{2}}} &&& \cf{n}{\xm{n}{n+3}}~,
}
where the map \w{\hcof'^{(n)}(\BBB{n}{2})} was described by Remark \ref{hcoftoB} (see also Proposition \ref{T=AcircB}).

Applying \w{\cof^{(n)}} to the left square and \w{\cof^{(n+1)}} to the right maps in the diagram of Proposition \ref{XF to RecVXF} yields
a commuting diagram, in which the vertical maps are the identity by Lemma \ref{lcfw=cf}, and whose top row is the bottom row of
\wref[.]{ababab} Thus \wref{ababab} is in fact:
\myssdg[\label{abababab}]{
\hcf{n}{\mathbb{C}^{(n)}X_1} \ar[d]_{\zeta_{\mathbb{C}^{(n)}X_1}}^{\simeq}
\ar[rrr]^{\hcf{n}{\AAA{n}{1} }} &&& \hcf{n}{\Mm{n}{2}}
\ar[d]^{\zeta_{\Mm{n}{2}}}_{\simeq} \ar[rrr]^{\hcof'^{(n)}(\BBB{n}{2})} &&& X_{n+3}
\ar[d]_{\simeq}^{r^{\xm{n}{n+3}}} \\
\cf{n}{\CCn{n}{X_1}} \ar[rrr]^>>>>>>>>>>>>>{\cf{n}{\AAAn{n}{1} }} &&&  \cf{n}{\Mn{n}{2}}
\ar[rrr]^{\cf{n}{\BBBn{n}{2}}} &&& \cf{n}{\xn{n}{n+3}}~,
}
where \w{\AAAn{n}{1}=\AAAn{n}{1}(V\XF)}  and \w[.]{\BBBn{n}{2}=\BBBn{n}{2}(V\XF)}
The result follows by Propositions \ref{T=AcircB} and \ref{nToda=cofb circ cofa}.
\end{proof}

By Theorem \ref{reduced rectifiyng} we conclude
\begin{corollary}
Given a strongly cofibrant $n$-th order cubical Toda system \w{\XF} of length \w{n+2} with \w{\Tm{n}{1}\XF} nullhomotopic, the underlying linear diagram \w{\Xs} is rectifiable.
\end{corollary}

We note the following variation for future reference:

\begin{prop}
Given a strongly cofibrant $n$-th order cubical Toda system\w{\XF} of length \w[,]{n+2}
with \w{f_{n+2} \circ \hcf{n}{\BBB{n}{1}\XF}} nullhomotopic, then
the underlying linear diagram \w{\Xs} is rectifiable.
\end{prop}
\begin{proof}
Setting \w[,]{\XFn:=V\XF} we showed in the proof of Theorem
\ref{reduced rectifiyng} how to rectify
\w{X_1 \stk{f_1} X_2 \stk{f_2}\dots X_{n+1} \stk{f_{n+1}} X_{n+2}}
by the sequence \wref[.]{rectify seq}

Because \w[,]{f_{n+2} \circ \hcf{n}{\BBB{n}{1}\XF} \sim \ast } we can choose a map \w{s:\cof(\bnn{n}{1}) \to X_{n+3}} fitting into
a commutative diagram:
$$
\xymatrix@R=20pt @C=40pt {
\hcf{n}{\Mm{n}{1}}  \ar[rr]^<<<<<<<<<<<<<<{\hcf{n}{\BBB{n}{1}\XF}} \ar[d]^{\simeq}_{\zeta_{\Mm{n}{1}}}  && X_{n+2} \ar[rr]^{f_{n+2}}
\ar[d]^{\simeq}_{r^{\xm{n}{n+2}}} && X_{n+3}\\
\cof(\ann{n-1}{1})\ar@{^{(}->}[rr]^{\bnn{n}{1}} && \cof(\Fn{n}{1})   \ar[rr]^{r^{\bnn{n}{1}}}
&& \cof(\bnn{n}{1}) \ar@{-->}[u]^{s}
}
$$
since the bottom row is a cofibration sequence.
We complete the rectification by replacing \w{f_{n+2}} with  \w[.]{s \circ r^{\bnn{n}{1}}}
\end{proof}

%
%
%
\sect{Passage from the recursive to the cubical definitions}\label{PRCD}

We now show how to lift a higher order recursive Toda system \w{\XFn} to a higher order cubical Toda system\w[,]{\XG}
using the diagrammatic description of \S \ref{nstwocubes}, and prove that this lifting process is inverse
(up to $\tos$) to the reduction \w[.]{\XF\mapsto V\XF}
Using Theorem \ref{def1=def2}, this implies that the two notions of Toda brackets are $\tos$-equivalent.

First recall the following special case of Lemma \ref{lfiltr}:

\begin{lemma}\label{full cofiber replacement}
Given an $n$-cube \w{\Ac} in $\Dd$,  let \w{\Ac'} denote its restriction to \w{\Ln{n-1}{}}  (see \S \ref{delln}).
Then any (strongly) cofibrant replacement for \w{\Ac'} can be extended to a strongly cofibrant replacement for \w[.]{\Ac}
\end{lemma}
\begin{prop}\label{fullreplacemen}
Given an $n$-th order recursive Toda system \w{\XFn} in $\Dd$ of length \w[,]{n+2}
we have a strongly cofibrant replacement
\mydiagram[\label{eqreplacer}]{
\CCm{n}{X_1} \ar[rrr]^{\AAA{n}{1} \XG} &&&  \Mm{n}{2}\XG
 \ar[rrr]^{\BBB{n}{2}\XG} &&& \xm{n}{n+3}\XG
}
for the sequence:
\mydiagram[\label{eqreplaced}]{
\CCn{n}{X_1} \ar[rrr]^{\AAAn{n}{1} \XFn} &&&  \Mn{n}{2}\XFn
 \ar[rrr]^{\BBBn{n}{2}\XFn} &&& \xn{n}{n+3}\XFn~,
}
\end{prop}

Here \wref{eqreplacer} corresponds to an $n$-th order cubical Toda system\w{\XG} by \S \ref{stwocubes}.

\begin{proof}
For each \w[,]{1\leq m\leq n} we have compatible strongly cofibrant replacements \w{\vartha{\CCm{m}{X_1}} : \mathbb{C}^{(m)}X_1 \to \CCn{m}{X_1}}
(see \S \ref{egva} and \S \ref{lnatwe}).
We extend the last of these, \w[,]{\vartha{\CCm{n}{X_1}}} to a strongly cofibrant replacement for \w[,]{\AAAn{n}{1}\XFn : \CCn{n}{X_1} \to \Mn{n}{2}\XFn }
defining successive extensions of \w{\vartha{\CCm{m}{X_1}}} by induction on \w[,]{m\geq 2} as follows:

For \w[,]{m=2} the diagram:
\myadiagnum[\label{diagramnum1}]{
X_1  \ar@{^{(}->}[d]_<<<{i^{X_1}} \ar@{^{(}->}[r]^{i^{X_1}} \ar[rrd]^>>>>>>>>{f_1} & CX_1 \ar@{^{(}->}[d]
\ar[rrd]^{Cf_1} & & \\
CX_1 \ar@{^{(}->}[r]^{i^{CX_1}} \ar[rrd]_{F_1:=\Fn{1}{1}} & C^2 X_1
& X_2 \ar[d]^<<<{f_2} \ar@{^{(}->}[r]_{i^{X_2}}
&  C X_2  \\
& & X_3  &  }
is a strongly cofibrant replacement for
\myadiagnumm[\label{diagramnum2}]{  X_1  \ar@{^{(}->}[d] \ar[r] \ar[rrd]^>>>>>>>>{f_1} & \ast \ar[d]
\ar[rrd]^{} & & \\
CX_1 \ar[r]^{} \ar[rrd]_{\Fn{1}{1}} & C \ee{}{X_1} & X_2
\ar[d]^<<<{f_{2}} \ar[r]  &  \ast  \\
& & X_3
&  }

 Using Lemma \ref{full cofiber replacement},  we can complete it to a strongly cofibrant replacement
for \w{\AAAn{2}{1}\XFn} (see \wref[).]{eqtwoone}

In the induction step, assume given a strongly cofibrant \w{\XG^{(m,m+1)}} (in the notation of \S \ref{order cahin}) and a weak equivalence \w[,]{\RR{m}{} : \Mm{m}{2}\XG \to \Mn{m}{2}\XFn} making
\mydiagram[\label{aa}]{
\CCm{m}{X_1}
\ar[d]_{\vartha{\CCm{m}{X_1}}}^{\simeq} \ar[rrrr]^>>>>>>>>>>>>>>>>{\AAA{m}{1}(\XG^{(m,m+1)})}  &&&& \Mm{m}{2}(\XG^{(m,m+1)}) \ar[d]^{\RR{m}{}}_{\simeq}  \\
\CCn{m}{X_1} \ar[rrrr]^{\AAAn{m}{1}\XFn} &&&& \Mn{m}{2}\XFn
}
commute. Thus \w{\AAA{m}{1} (\XG^{(m,m+1)}) } is a strongly cofibrant replacement for \w[.]{\AAAn{m}{1} \XFn }

We need to extend \wref{aa} to \w{\AAA{m+1}{1}(\XG^{(m+1,m+2)})} and \w[.]{\RR{m+1}{}}
Using Lemmas \ref{paCCnX} and \ref{npa} and diagram \wref[,]{diagrampa}
the map of \wwb{m+1}cubes \w{\AAAn{m+1}{1}\XFn} is given by a diagram of $m$-cubes:
\mydiagram[\label{ab}]{
\CCn{m}{X_1}\ar[d]_{(i^{\ee{m}{X_1}}\circ\, r^{i^{\ee{m-1}{X_1}}})\sb{\ast} }
\ar[rrr]^{\AAAn{m}{1}\XFn}  &&& \Mn{m}{2}\XFn \ar[d]^{\BBBn{m}{2}\XFn}  \\
\Fin{m}{C\ee{m}{X_1}} \ar[rrr]^{\Fin{m}{\Fn{m+1}{1}}} &&& \Fin{m}{\cof(\Fn{m}{2})}
}
(see \S \ref{Fin}), so we have a strongly cofibrant replacement for \w{\AAAn{m+1}{1} \XFn}  restricted to the upper left corner of \wref[,]{ab} back to front in the following diagram:
\mysdiag[\label{abc}]{
\CCm{m}{X_1} \ar[rrrd]^{\simeq}_{\vartha{\CCm{m}{X_1}}}   \ar@{^{(}->}[d]_{i^{\CCm{m}{X_1}}}
\ar[rrr]^>>>>>>>>>>{\AAA{m}{1}\XG}
&&& \Mm{m}{2}(\XG^{(m,m+1)})  \ar[rrrd]_{\simeq}^{\RR{m}{}} \\
C\CCm{m}{X_1} \ar[drrr]^{\simeq}_{(r\sp{C\CCm{m}{X_1}})\sb{\ast}}  &&&   \CCn{m}{X_1} \ar[d]_{}  \ar[rrr]_{\AAAn{m}{1}\XFn}  &&& \Mn{m}{2}\XFn \\
&&&   \Fin{m}{C\ee{m}{X_1}}
}

We want to extend \wref{abc} to a full strongly cofibrant replacement for \wref{ab} by extending the back of \wref{abc} as follows:
\myudiag[\label{abcd}]{
\CCm{m}{X_1}    \ar@{^{(}->}[d]_{i^{\CCm{m}{X_1}}} \ar[rrrrr]^>>>>>>>>>>>>>>>>>>>{\AAA{m}{1}(\XG^{(m,m+1)})}
&&&&& \Mm{m}{2}(\XG^{(m,m+1)}) \ar[d]^{\BBB{m}{2}(\XG^{(m,m+2)})} \\
C\CCm{m}{X_1} \ar[rrrrr]^<<<<<<<<<<<<<<<<<<{\GGG{m+1}{1}(\XG^{(m+1,m+2)})}  &&&&&    \xm{m}{m+3}(\XG^{(m+1,m+2)})
}
where \w{{\GGG{m+1}{1}(\XG)}} is a diagrammatic nullhomotopy as in \wref[,]{eqdiagnull}
in two steps:

\begin{itemize}
\item First, we have canonical choices for the right and bottom
maps in \wref{abcd} on the subdiagrams indexed by \w{\Ln{m-1}{}}
(that is, all but the last vertex of the $m$-cube). This is
because, by Definition \ref{BBB}, the restriction to
\w{\Ln{m-1}{}} of the map of $m$-cubes
\w{\BBB{m}{2}(\XG^{(m,m+2)})} is completely determined by
\w[.]{\Mm{m}{2}(\XG^{(m,m+1)})} Similarly, by Definition
\ref{diagnull}, the restriction to \w{\Ln{m-1}{}} of the
diagrammatic nullhomotopy \w{\GGG{m+1}{1}(\XG^{(m+1,m+2)})} is
determined by \w[]{\XG^{(m,m+1)}} (see Diagram \wref{eqmapcorners}
below).

Note that the extension of the back-to-front (weak equivalence) map of squares in \wref{abc} extends trivially on \w[,]{\Ln{m-1}{}} since all new targets are the zero object and all sources are cones.
\item Finally, again thinking of each square of $m$-cubes as an \wwb{m+2}cube, we can use
Lemma \ref{full cofiber replacement}
to extend to the last vertex \w{(1,\dotsc,1)} of \w[,]{\In{m+2}{}} obtaining a full strongly cofibrant
replacement \wref{abcd} for \wref[.]{ab} In particular, we have \w{\XG^{(m+1,m+2)}} extending \w{\XG^{(m,m+1)}} and a weak equivalence \w[.]{\RR{m+1}{}}
\end{itemize}

We thus constructed the left half of \wref{eqreplacer} and mapped it to \wref[.]{eqreplaced} For the right half,
we want the dotted maps in
\mydiagram[\label{eqtightrep}]{
  \Mm{n}{2}\XG \ar[d]^{\RR{n}{}}_{\simeq} \ar@{-->}[rrr]^{\BBB{n}{2}\XG} &&& \xm{n}{n+3}\XG \ar@{-->}[d]^{\simeq} \\
 \Mn{n}{2}\XFn \ar[rrr]^{\BBBn{n}{2}\XFn} &&& \xn{n}{n+3}\XFn
}
By definition we have iterated cones for \w{\xm{n}{n+3}\XG} and inclusions for \w{\BBB{n}{2}\XG} in all but the last vertex \w{(1,\dotsc,1)}
(see diagram \wref{abeqmapcorners} below). Since \w[,]{\xn{n}{n+3}\XFn =\Fin{n}{\cof(\Fn{n}{2})}}  which is trivial in all but the last vertex,
we have a canonical choice for \wref{eqtightrep} on \w{\Ln{n}{}}
and use Lemma \ref{full cofiber replacement} to complete \wref[.]{eqtightrep}
\end{proof}
\begin{example}\label{exfullreplacment}
Assume given a third order recursive Toda system \w{\XFn}  of length $5$ in $\Dd$. We start with the case  \w{n=2}
by completing \wref{diagramnum1} as in the proof above.
This will serve as the top $3$-cube in the following strongly cofibrant replacement for the next step.

Here the left half of \wref{eqreplacer} appears as a back to front map of $3$-cubes \w{\AAAn{3}{1}\XFn :\CCn{3}{X_1}\to \Mn{3}{2}\XFn }
(where the missing final vertex of \w[,]{\Mn{3}{2}\XFn} to be called \w[,]{X\sb{5}'} is indicated by $?$):
\myrrdiag[\label{eqmapcorners}]{
& X_1 \ar@{^{(}->}[ldd] \ar@{^{(}->}[d]_<<<{i^{X_1}} \ar@{^{(}->}[r]^{i^{X_1}}
\ar[rrd]^>>>>>>>>>>>>>{f_1} & CX_1 \ar@{^{(}->}[ldd]  \ar@{^{(}->}[d]^<<<{}
\ar[rrd]^{Cf_1} & & \\
& CX_1 \ar@{^{(}->}[ldd] \ar@{^{(}->}[r]^{i^{CX_1}}
\ar[rrd]_>>>>>>>>>>>>>>>{\Fn{1}{1}} & C^2 X_1 \ar@{^{(}->}[ldd]
 \ar[rrd]^<<<<<<<<<<<{\Gm{2}{1}} & X_2\ar@{^{(}->}[ldd]
 \ar[d]^<<<{f_2} \ar@{^{(}->}[r]_{i^{X_2}}
&  C X_2 \ar@{^{(}->}[ldd] \ar[d]^{G_2} \\
CX_1 \ar[rrd]^>>>>>>>>{Cf_1} \ar@{^{(}->}[d]_<<<{C(i^{X_1})} \ar@{^{(}->}[r]
& C^2X_1 \ar[rrd]^>>>>>>>>>>>>>>>>>>>>>{C^2f_1}   \ar@{^{(}->}[d]^<<<{}  &  &
X_3 \ar@{^{(}->}[ldd]  \ar[r]_{g_3} & X'_4 \\
C^2X_1 \ar[rrd]_{C\Fn{1}{1}} \ar@{^{(}->}[r]^{C(i^{CX_1})} & C^3 X_1
&    CX_2 \ar[d]_{Cf_2}\ar@{^{(}->}[r]_{C(i^{CX_2})} &  C^{2}X_2   \\
& &  CX_3  & ?
}
Using Lemma \ref{full cofiber replacement}, we can complete this to a strongly cofibrant replacement for full map of $3$-cubes
\w[.]{\AAAn{3}{1}\XFn :\CCn{3}{X_1}\to \Mn{3}{2}\XFn }

Therefore, we have a strongly cofibrant replacement (back to front) for all but the last vertex of
\w[:]{\BBBn{3}{2}\XFn :\Mn{3}{2}\XFn \to \xn{3}{6}\XFn }
\myrrdiag[\label{abeqmapcorners}]{
& X_2 \ar@{^{(}->}[ldd] \ar@{^{(}->}[d]^{f_2} \ar@{^{(}->}[r]^{i^{X_2}}
\ar@{^{(}->}[rrd] & CX_2 \ar@{^{(}->}[ldd]  \ar[d]^<<<{G_2}
\ar@{^{(}->}[rrd]  & & \\
& X_3 \ar@{^{(}->}[ldd] \ar@{^{(}->}[r]^{g_3}
\ar@{^{(}->}[rrd] & X'_4 \ar[ldd]^<<<<<<<<{g_4}
 \ar@{^{(}->}[rrd] & C X_2\ar@{^{(}->}[ldd]
 \ar[d]^<<<{Cf_2} \ar@{^{(}->}[r]_{Ci^{X_2}}
&  C^2 X_2 \ar@{^{(}->}[ldd] \ar[d]^{CG_2} \\
CX_2 \ar@{^{(}->}[rrd] \ar[d]_{Cf_2} \ar@{^{(}->}[r]
& C^2X_2 \ar@{^{(}->}[rrd]   \ar[d]^<<<{\Gm{2}{2}}
& & CX_3 \ar@{^{(}->}[ldd]  \ar[r]_<<<<{Cg_3} & CX'_4   \\
CX_3 \ar@{^{(}->}[rrd] \ar[r]^{G_3} & X'_5
&    C^2X_2 \ar[d]_{C^{2}f_2}\ar@{^{(}->}[r] &  C^{3}X_2   \\
& &  C^2X_3  & ?
}
Again we use Lemma \ref{full cofiber replacement} to complete the missing vertex $?$.
\end{example}
\begin{remark}
Note that, given a recursive higher Toda system \w[,]{\XFn} the cubical higher Toda system \w{\XG} of Proposition \ref{fullreplacemen}
is not functorial, since we choose liftings in each induction step. However,
if \w{\XFn=V\XF} for some $n$-th order cubical Toda system\w{\XF} (see \S \ref{indchcx}),
then by Proposition \ref{XF to RecVXF} we can choose \w{\XG} to be \w{\XF} itself.
Thus the construction of \w{\XG} can be thought of as a left inverse for the functor \w[.]{V}

The following Proposition provides the other direction:
\end{remark}
\begin{prop}\label{VXG tosr XFn}
Given an $n$-th order recursive Toda system \w{\XFn} in $\Dd$,  with \w{\XG} as in Proposition \ref{fullreplacemen}, we have  \w[.]{V\XG \tosr \XFn}
\end{prop}

\begin{proof}
By the construction of \w{\XG} in the proof of Proposition \ref{fullreplacemen} we see that
$$
\xymatrix@R=25pt  {
\CCn{n}{X_1} \ar[rrr]^{\AAAn{n}{1} \XFn} &&&  \Mn{n}{2}\XFn
 \ar[rrr]^{\BBBn{n}{2}\XFn} &&& \xn{n}{n+3}\XFn
}
$$
is $\tos$-equivalent to \wref[,]{eqreplacer} which is $\tos$-equivalent by
 Proposition \ref{XF to RecVXF} to
$$
\xymatrix@R=25pt @C=35pt {
\CCn{n}{X_1} \ar[rr]^<<<<<<<<<<<{\AAAn{n}{1}(V\XG)} &&  \Mn{n}{2}(V\XG)
 \ar[rrr]^{\BBBn{n}{2}(V\XG)} &&& \xn{n}{n+3}(V\XG)
}
$$
\end{proof}

We thus obtain our main result, the converse of Theorem \ref{def1=def2}:

\begin{thm}\label{Rev def1=def2}
If \w{\XFn} is an $n$-th order recursive Toda system  in $\Dd$, and \w{\XG} is as in Proposition
\ref{fullreplacemen}, then \w[.]{\Tm{n}{1}\XG \tos \Tn{n}{1}\XFn }
\end{thm}
\begin{proof}
By Theorem \ref{def1=def2}  \w{\Tm{n}{1}\XG \tos \Tn{n}{1}(V\XG) } and by Proposition
\ref{VXG tosr XFn}  and Corollary \ref{Tn=Tn} we get that
\w[.]{\Tn{n}{1}(V\XG) \tos \Tn{n}{1}\XFn }
\end{proof}
In fact we can show that we have a commutative diagram:
$$
\xymatrix@R=25pt {
L^n X_1  \ar[rrr]^{\Tm{n}{1}\XG}
 \ar[d]_{\zeta_{\mathbb{C}^{n}X_1}}^{\simeq}  &&& X'_{n+3}
\ar[d]^{\simeq} \\
{\widetilde{\Sigma}}^{n} X_1
 \ar[rrr]_{\Tn{n}{1}\XFn }
&&& \cof^{(n)}(\xn{n}{n+3}).
}
$$
\end{document}